\documentclass[12pt]{amsart}

\usepackage[english]{babel}

\usepackage{times,amsfonts,amsmath,amssymb,dsfont} 
\usepackage{pdfsync}
\usepackage{graphicx} 
\usepackage{mathabx}

\usepackage{cases}

\usepackage{soul}
\setstcolor{red}

\usepackage{color}

\definecolor{gr}{rgb}   {0.,   0.69,   0.23 }
\definecolor{bl}{rgb}   {0.,   0.5,   1. }
\definecolor{mg}{rgb}   {0.85,  0.,    0.85}
\definecolor{or}{rgb}   {0.9,  0.5,   0.}

\definecolor{webred}{rgb}{0.75,0,0}
\definecolor{webgreen}{rgb}{0,0.75,0}
\usepackage[citecolor=webgreen,colorlinks=true,linkcolor=webred]{hyperref}

\newtheorem{theorem}{Theorem}[section]
\newtheorem{proposition}[theorem]{Proposition}
\newtheorem{lemma}[theorem]{Lemma}
\newtheorem{corollary}[theorem]{Corollary}

\theoremstyle{definition}

\newtheorem{notation}[theorem]{Notation}

\theoremstyle{remark}
\newtheorem{remark}[theorem]{Remark}

\newcommand{\ba}{\begin{array}}
\newcommand{\ea}{\end{array}}

\newcommand{\Bk}{\color{black}}

\newcommand{\N}{\mathbb{N}}
\newcommand{\R}{\mathbb{R}}
\newcommand{\C}{\mathbb{C}}
\newcommand{\Z}{{\mathbb Z}}

\newcommand{\rd}{{\mathrm d}}
\newcommand{\one}{\mathds{1}}

\newcommand{\cE}{\mathcal{E}}
\newcommand{\cF}{\mathcal{F}}

\newcommand{\cT}{\mathcal{T}}
\newcommand{\cU}{\mathcal{U}}

\newcommand{\bel}{\begin{equation} \label}
\newcommand{\ee}{\end{equation}}

\newcommand{\dP}{\mathbb{P}}

\newcommand{\qm}{\mathrm{qm}}

\newcommand\dist{\operatorname{dist}}

\newcommand\supp{\operatorname{supp}}
\newcommand\curl{\operatorname{curl}}

\newcommand\Dom{\operatorname{Dom}}
\newcommand\Id{\operatorname{\mathbb{I}}}

\begin{document}\setul{2.5ex}{.25ex}
\title[]{Spectrum of the Iwatsuka Hamiltonian at thresholds}
\author[P.\ Miranda]{Pablo Miranda}\address{Departamento de Matem\'atica y Ciencia de la Computaci\'on, Universidad de Santiago de Chile, Las Sophoras 173. Santiago, Chile.}\email{pablo.miranda.r@usach.cl}
\author[N.\ Popoff]{Nicolas Popoff}
\address{Universit\'e de Bordeaux, IMB, UMR 5251, 33405 TALENCE cedex, France}
\email{Nicolas.Popoff@math.u-bordeaux1.fr}
\begin{abstract}
We consider the  bi-dimensional Schr\"odinger operator with unidirectionally constant magnetic field, $H_0$, sometimes known as the ``Iwatsuka Hamiltonian''. This operator is  analytically fibered, with band functions converging to finite limits  at infinity. We first obtain the asymptotic behavior of the band functions and its derivatives. Using this results we give estimates on the current and on the localization of states whose energy value is close to a  given \emph{threshold} in the spectrum of $H_0$. In addition, for   non-negative  electric perturbations $V$ we  study  the spectral density of $H_0\pm V,$ by considering  the   Spectral Shift Function associated  to the operator  pair $(H_0\pm V,H_0)$.  
We compute the   asymptotic behavior of  the  Spectral Shift Function at the thresholds, which are the only points where it can grows to infinity.

\end{abstract}
\maketitle

\section{Introduction}
\subsection{Context and motivation}

Let $b$ be a scalar magnetic field in $\R^2$ that 
is translationally invariant in the sense that it does not depend on one of the spatial variables: $b(x,y)=b(x)$. Let $a:\R\to \R$ be the  function  $a(x)=\int_0^xb(t)\rd t$. Then  the potential $A(x):=(0,a(x))$ satisfies $\curl A=b$. 
In this article we study 
 the magnetic Schr\"odinger operator
$$H_{0}:=(-i\nabla-A)^2=-\partial_{x}^2+(-i\partial_{y}-a(x))^2$$
acting in $L^2(\R^2)$. 

In \cite{Iwa85}  Akira Iwatsuka considered this class of Hamiltonians  in order to show new examples of magnetic Schr\"o\-dinger operators with purely absolutely continuous spectrum.   This model has many interesting  properties,   and some of them have been well studied in the past. For instance,  in the physics literature it can be seen as an open quantum waveguide in the sense that the presence of such magnetic field generates transport for a spinless particle in a plane (see \cite{PeeRej00} for  references). Rigorous propagation  properties of this model are described for example  in \cite{manpu,lww,dgr,hissocc2}. Additionally,  the spectral properties of $H_0$ have also  been studied. In this sense, one important conjecture on the  spectrum of $H_0$  is that it should be purely  absolutely continuous as long as $b$ is non-constant. However, although it  have been found different conditions on $b$   that ensure this property, the general result remains unproved  (see \cite{ Iwa85, manpu, exkov, matej}).  

In order to describe the problems that  we propose to study in this article it is necessary  to write  a well known decomposition of the operator $H_0$.

\emph{Fiber decomposition of $H_0$.}
Let ${\mathcal F}$ be  the partial Fourier transform  with respect to the translationally invariant variable $y\in \R$: \Bk
$$  (\mathcal{F}u)(x,k)=\frac{1}{(2\pi)^{1/2}}\int_{\R}e^{-iky}u(x,y)\,\rd y,\quad \mbox{for}\,\,u \,\in\, C_0^\infty(\R^2).$$
Then
\bel{direc} \mathcal{F} H_0\mathcal{F}^*=\int_{\R}^\oplus h(k)\,\rd k,\ee
where $h(k)$ is a self-adjoint operator acting in $L^2(\R)$, defined by
\bel{sep8}h(k)=-\frac{d^2}{\rd x^2}+(a(x)-k)^2, \quad k \,\in\, \R.\ee
For any $k \in  \R$, under hypotheses \eqref{b_bound} below, $h(k)$ has  compact resolvent, {  therefore  its spectrum  is discrete, and moreover, it is simple}. We denote the increasing sequence of eigenvalues by $\{E_n(k)\}_{n=1}^\infty$. For any $n \in \N$, \Bk  the \emph{band function} $E_n(\cdot)$  is analytic as a function of $k \in \R$ \cite{Iwa85}.

In a broad sense, the results of  this article are valid for magnetic fields that satisfy: 
\bel{b_bound}
  \begin{array}{llr}
  \rm{a})& b \in C^{\infty}(\R).\\
 \rm{b})&  b \ \mbox{is  increasing}.\\
 \rm{c})& {\lim_{x\to  \pm\infty}b(x)=b_\pm, \quad \mbox{for } b_+>b_->0.}\\
\end{array}
\ee
Even more, each one  of these hypotheses can be relaxed, but the results are much  easier to read for these kind in magnetic fields. 

Set $\underline{\mathcal E}_n:=b_-(2n-1)$ and $\overline{\mathcal E}_n:=b_{+}(2n-1)$, then, condition \eqref{b_bound}  implies that   $\underline{\mathcal E}_n\leq E_n(k)\leq \overline{\mathcal E}_n$ for all $k \in \R.$ Furthermore, $E_n(\cdot)$ is strictly increasing for any $n\in\N$, and   $\lim_{k \to +\infty} E_n(k)=\overline{\mathcal E}_n$;    $\lim_{k \to -\infty} E_n(k)=\underline{\mathcal E}_n$  for all $n \in \N$, see \cite{Iwa85, manpu}.  In particular, the spectrum of $H_0$, $\sigma(H_0)$, is purely absolutely continuous and  
\bel{oct7}\sigma(H_0)=\bigcup_{n \geq1}\overline{E_n(\R)}=\bigcup_{n\geq 1}[\underline{\mathcal E}_n,\overline{\mathcal E}_n].\ee
Moreover, using decomposition \eqref{direc} and the monotonicity of $E_n$ it is possible to see that the multiplicity of the spectrum of $H_0$ changes at any point in 
the set of points $\{ \underline{\cE}_{n},  \overline{\cE}_{n}\}_{n=1}^\infty$. This set of points will be refered as \emph{thresholds} in  $\sigma(H_0)$. We denote it by $\cT_{H_0}$. 
  
\subsection{Description of the main results of the article}  
 The first step in the physical description of our magnetic system  at energies near thresholds, requires to precise the behavior of the band functions when  $k\to\pm\infty$. Therefore,  our first task in this article is to describe qualitatively this behavior. We will do this for magnetic fields converging to their limits as power decaying functions, with similar hypotheses on its derivatives (see \eqref{b_bound2}). In Theorem \ref{T:asymptband} we  obtain    the first three asymptotic terms of $k\mapsto E_n(k)$, which are  roughly  given by $(2n-1)b(\frac{k}{b_{+}})$. We also give  the  asymptotic behavior  of the first and second derivative, which can be formally  obtained as the derivatives of the asymptotic term of $E_n$. A byproduct of our analysis in  section \ref{sec_asymp_band} is  an estimate of the eigenfunctions by a linear combination of Hermite's functions.
  
In section \ref{S:bulkstates}, we exploit these results to obtain properties of functions localized in energy near a given threshold. This problematic arises from physical consideration: it consists in taking a quantum state of a given energy, and to study its evolution under the magnetic system considered. We are particularly interested in the propagation properties in the translationally invariant direction (that is the $y$ direction). As explained in \cite{manpu,hissocc2}, a state localized in energy away from thresholds bear a current which can be bounded from below. Moreover, it is localized in the $x$ variable, i.e. it decays in the direction in which the magnetic field varies. On the contrary, when the energy interval contains a thresholds, no lower bound on the current is available. Inspired by the approach of \cite{hispopsoc}, we study the {\it bulk component} of a quantum states localized in energy near a given threshold. Using the asymptotics of the band functions and its derivatives, we are able to provide a quantitative control on the  current carried by the given state,  as the distance from its energy to the threshold goes to 0 (see Proposition \ref{corocurrent}). Further, the asymptotics of the eigenfunctions of the fiber operator allow us to prove that a state localized in a energy near a threshold is concentrated  in a zone where the magnetic field is close to its limit. See Theorem \ref{theo_loca}. 

In sections \ref{secSSF}--\ref{proofasym_ssf} we will perturb the operator $H_0$, and consider the  spectral density of the resulting   operator. More specifically, we will take a non-negative decaying electric potential $V$  and will define $H:=H_0\pm V$.  Under these hypotheses the essential spectrum of $H$ coincide with the essential spectrum of $H_0$. However,   discrete eigenvalues may appear. Counting the number of such eigenvalues has been considered for $b$ constant and non-constant (for the Landau case, $b=$ constant, see for instance   \cite{rai1,raiwar,pushroz2,roz}, and for Iwatsuka \cite{shi2, DoHiSo14, Mir16}). In this article we  study  an  extension of this problem to  the continuous  spectrum of $H$, namely, we describe the properties of the spectral shift function (SSF) for the operator pair $(H,H_0).$ (The connection between the eigenvalue counting function and the SSF is given by  \eqref{15_sep_16}).


In Theorem \ref{thssf} we prove that the SSF   is continuous everywhere except at the thresholds in $\cT_{H_0}$
and the eigenvalues of $H$. Moreover, we show that it  is a bounded function in compact subsets of $\R\setminus\cT_{H_0}$. 
Subsequently,  Theorem \ref{thVcompact} states that the SSF is bounded at the thresholds as well, as long as the rate of convergence of $V$ to zero is large compared to the rate of convergence of $b(x)$ to $b_\pm$\Bk. On the contrary, in Theorem \ref{thssf2} we  describe  the behavior of the SSF at the thresholds, for certain potentials $V$ that decay moderately at infinity. In this case, the SSF is unbounded 
as a fixed threshold is approached. This result is given by  the semiclassical formula \eqref{asym_ssf} which relates the behavior of the SSF with a volume in phase space determined by the potential $V$, and partly extend some of the results obtained in the above articles on the Landau and Iwatsuka Hamiltonians.

The approach  we use here to study the SSF comes somehow from \cite{rai1}, where the discrete spectrum of the Landau Hamiltonian was described. This approach was also used to analyze  the discrete spectrum of the Iwatsuka Hamiltonian in \cite{Mir16}. However,  in the aforementioned articles it was not necessary to make a detailed analysis of the corresponding  band functions; they are constants in the Landau case and some variational estimates are enough for the Iwatsuka case. There are  other models where the  asymptotics of band functions was necessary   to make the spectral analysis, the closest one being a half-plane submitted to a  constant magnetic field (\cite{bmr2,hispopsoc,brumi1}). 
One of the differences with  our situation is that we consider a whole class of varying magnetic field, and in these references the magnetic field is constant, allowing, for instance, the use of special functions to study the fiber operator.

 In this article, we will  state almost  all of our results for the thresholds $\{\overline{\cE}_n\}_{n=1}^\infty$.  Since this are thresholds corresponding to the fact that $b(x)\to b_{+}$ as $x\to +\infty$,  we only write the assumptions on $b$ for $x\to+\infty$. It is clear that the same assumptions as $x\to -\infty$ would lead to symmetric results for the  thresholds  $\{\underline{\cE}_n\}_{n=1}^\infty$. 

\section{Asymptotics of the band functions}\label{sec_asymp_band}

We will  assume:
\bel{b_bound2}
  \begin{array}{llr}
  \rm{a})& \exists\, x_{0}\in \R,\ \forall p\in \N, \ \forall x\geq x_{0}, \quad  {b}^{(p)}(x)\neq0 .\\
  \rm{b})&b'(x) =o(b_+-b(x)),\, \textrm{and} \,\forall p\in \N,\  {b}^{(p+1)}(x)=o({b}^{(p)}(x));\quad  x\to+\infty. &\\
  \rm{c})& ({b}')^2(x)=o({b}''(x)); \quad x \to +\infty. \\
  \end{array}
\ee
These hypothesis will be naturally satisfied for  magnetic fields that converge to $b_{+}$ as a negative power of $x$. Therefore, in the following sections we will exploit our results for this kind of magnetic fields (see assumption \eqref{E:HypB} below).  
  \begin{notation}
  \label{N:akco}
  Denote by 
$x_{k}:=a^{-1}(k)$, the point where the potential $(a(x)-k)^2$ attains its minimum; $b_{k}:=b(x_{k})$ and for $p\geq1$, $b_{k}^{(p)}:=b^{(p)}(x_{k})$. 

For $\sigma>0$  and $\tau>0$, 
\Bk define the function 
\begin{equation}
\label{E:defepsilon}
\epsilon_{\sigma,\tau}(k):=(b_{k}')^2+\sup_{x\in\Bk({\sigma} k,+\infty)}\{|b'b''(x)|+|b^{(3)}(x)|\}+e^{-\tau k^2}.
\end{equation}
For readibility, we will write $\epsilon(k):=\epsilon_{\sigma,\tau}(k)$.
\end{notation}

\begin{theorem}
\label{T:asymptband}Assume that  $b$ satisfies \eqref{b_bound} and  \eqref{b_bound2},  
then for any $n\in \N$, there exist $\sigma>0$ and $\tau>0$ such that 
\begin{equation}
\label{E:asymptotEn}
E_{n}(k)=b_{k}\Lambda_{n}+\gamma_{n}b_{k}^{-1}b_{k}''+O( \epsilon(k)),  \quad k\to \infty,\end{equation}
with $\Lambda_n:=2n-1$ and  $\gamma_{n}:=\tfrac{1}{4}(2n^2-2n+1)$.

Moreover, the derivatives have an asymptotic expansion that is the formal derivative of $E_{n}$, up to the remainder:
\begin{equation}
\label{E:asymptotEnp}
E_{n}'(k)=b_{k}^{-1}\Lambda_{n}b_{k}'+O( \epsilon(k)),\quad k\to \infty,\end{equation}
and 
\begin{equation}
\label{E:asymptotEpp}
E_{n}''(k)=b_{k}^{-2}\Lambda_{n}b_{k}''+O( \epsilon(k)),\quad k\to \infty.\end{equation}
In the above asymptotics, the remainders depends on $n$. 
\end{theorem}
\begin{remark}
It is possible to precise the range of values of the constants involved in Definition \eqref{E:defepsilon}: The theorem holds true at least for all constants $(\sigma,\tau)\in (0,\frac{1}{b_{+}}-\frac{\sigma_{0}}{(b_{+})^{1/2}})\times (0,\frac{\sigma_{0}^2}{2})$, where $\sigma_{0}\in (0,\frac{1}{(b_{+})^{1/2}})$.
\end{remark}

We exploit Theorem \ref{T:asymptband} in the following model case: 
\begin{equation}
\label{E:HypB}
\exists x_{0} >0, \forall x \geq x_{0}, \quad b_{+}-b(x)=\frac{1}{x^{M}} \ \ \mbox{with} \ \ M>0.
\end{equation}
Here we again impose a condition much stronger than necessary in order to write the results of the following sections concisely.

\begin{corollary}
\label{C:asymptBexplicit}
Assume \eqref{b_bound} and \eqref{E:HypB}. Set $\eta=\min(2M+2,M+3)$, then, as $k\to+\infty$: 
$$E_{n}(k)-\overline{\mathcal E}_n=-\frac{\Lambda_{n}b_{+}^{M}}{k^{M}}+O\left(\frac{1}{k^{M+2}}\right).$$
$$E_{n}'(k)=M\frac{\Lambda_{n}b_{+}^{M}}{k^{M+1}}+O\left(\frac{1}{k^{\eta}}\right) \ \ \mbox{and} \ \ E_{n}''(k)=-M(M+1)\frac{\Lambda_{n}b_{+}^{M}}{k^{M+2}}+O\left(\frac{1}{k^{\eta}}\right).$$
\end{corollary}\Bk
Although this corollary is a direct consequence of Theorem \ref{T:asymptband}, it can also be seen as the first terms of the asymptotics provided by the harmonic approximation for a semi-classical Schr\"odinger operator (see Remark \ref{R:semiclassical}).

\subsection{Construction of quasi-modes}
\label{SS:QM}
 In order to prove Theorem \ref{T:asymptband} we will apply the standard procedure of constructing quasi-modes, for  $h(k)$  in \eqref{sep8} which  is a Sturm-Liouville operator. The potential $(a(x)-k)^2$ vanishes only at the point $x=x_{k}=a^{-1}(k)$. 
Near this point, its  formal Taylor expansion is $b_{k}^2(x-x_{k})^2+O(x^3)$, where the remainder depends on $k$. Therefore, guided by the semiclassical approximation, we perform the change of variables  
\bel{E:CV}
t=b_{k}^{1/2}(x-x_{k}),
\ee and the operator $h(k)$ becomes
\bel{10apr17}b_{k}\left(-\partial_{t}^2+w(t,k)\right), \quad t\in \R,\ee 
where 
$$w(t,k):=b_{k}^{-1}(a(b_{k}^{-1/2}t+x_{k})-k)^2.$$
We compute the derivatives of $w$, with the notation $x=b_{k}^{-1/2}t+x_{k}$:
\begin{equation*}
\begin{aligned}
&w'(t,k)=2b_{k}^{-3/2}(a(x)-k)b(x) \\
&w^{(2)}(t,k)=2b_{k}^{-2}(b(x)^2+(a(x)-k)b'(x)))\\
&w^{(3)}(t,k)=2b_{k}^{-5/2}(3b(x)b'(x)+(a(x)-k)b''(x))\\
&w^{(4)}(t,k)=2b_{k}^{-3}(3b'(x)^2+4b(x)b''(x)+(a(x)-k)b^{(3)}(x)) \ .
\end{aligned}
\end{equation*}
In consequence, the Taylor expansion of the normalized potential $w$  near $t=0$ writes 
\bel{11apr17a}w(t,k)=t^2+\alpha_{1}(k)t^3+\alpha_{2}(k)t^4+R_{4}(t,k),\ee
with 
\begin{equation}
\label{E:defalpha}
\alpha_{1}(k)=b_{k}^{-3/2}b_{k}' \ \ \mbox{and} \ \ \alpha_{2}(k)=\frac{1}{4}b_{k}^{-3}(b_{k}')^2+\frac{1}{3}b_{k}^{-2}b_{k}''.
\end{equation}

\begin{remark}
\label{R:semiclassical}
The tools used  here are closed to those of the semi-classical harmonic approximation (\cite{DiSj99}). Notice that by trying the scaling $X=x_{k}x$ and setting $\hbar=x_{k}^{-2}$, we are led (up to factor) to an operator of the form $-\hbar^2\partial_{X}^2+W(X,h)$, where the potential $W$ has a unique non degenerate minimum, and a Taylor expansion near the minimum in powers of $X$, depending on $h$ but also on  $b^{(p)}_{k}$. The hypotheses \eqref{b_bound2} express then  conditions of the potential $W$ so that this operator enters indeed in the framework of harmonic approximation. Under condition \eqref{E:HypB}, the potential $W$ has an expansion with powers of $X$, and with various powers of  $\hbar$  (and therefore in $k^{-1}$), but due to the different scales, it is not clear that a full asymptotic expansion in power of $k^{-1}$ exists. 
\end{remark}
From \eqref{10apr17} and \eqref{11apr17a} w\Bk e are  led to compute the spectrum of the formal operator $h_0+w_1+w_2+R_{4}$ with 
\bel{13mar17} h_0=-\partial_{t}^2+t^2, \ \ w_1=\alpha_{1}t^3, \ \  w_2=\alpha_{2}t^4, \ \ R_{4}(t,k)=w(t,k)-(t^2+\alpha_{1}t^3+\alpha_{2}t^4).\ee
 We will consider these operators formally, and construct a quasi eigenpair of the form 
\begin{equation}
\label{E:anzatz}
(\mu_{0}+\mu_{1}+\mu_{2},\varphi_{0}+\varphi_{1}+\varphi_{2}), 
\end{equation}
where this ansatz is adapted to perturbation theory, as follows:
\subsubsection{The engine}
 Let $\Psi_{n}$ be the $n$-th normalized Hermite's function, starting from $n=1$ (see for example \cite{abst}).



\Bk
The order 0 term leads to $(h_0-\mu_{0})\varphi_{0}=0$, which is solved by taking  
$$\mu_{0}=\Lambda_{n} \ \ \mbox{and} \ \ \varphi_{0}=\Psi_{n}.$$
\subsubsection{First order}
We look for a pair $(\mu_{1},\varphi_{1})$ such that $(h_0-\mu_{0})\varphi_{1}=(\mu_{1}-w_1)\varphi_{0}$.  The Fredholm alternative writes $\langle (\mu_{1}-w_1)\Psi_{n},\Psi_{n} \rangle=0$, therefore, using the symmetry of $\Psi_{n}$  and the oddness of $w_{1}$, \Bk we get 
 $$\mu_{1}=0 \ \ \mbox{and} \ \ \varphi_{1}=-(h_0-\mu_{0})^{-1}(w_1\Psi_{n}).$$
Taking into account that 
\begin{equation}
\label{E:decomposetPsi}
t\Psi_{n}(t)=\sqrt{\frac{n-1}{2}}\Psi_{n-1}+\sqrt{\frac{n}{2}}\Psi_{n+1},
\end{equation}
we obtain
\begin{multline}
\label{E:decomposet3Psi}
t^3\Psi_{n}=2^{-3/2}\big(\sqrt{(n-1)(n-2)(n-3)}\Psi_{n-3}+3(n-1)\sqrt{n-1}\Psi_{n-1}
\\
+3n\sqrt{n}\Psi_{n+1}+\sqrt{n(n+1)(n+2)}\Psi_{n+3}\big).
\end{multline}
Hence
\begin{multline}
\varphi_{1}=-\alpha_{1}(k)2^{-5/2}\Big(-\tfrac{1}{3}\sqrt{(n-1)(n-2)(n-3)}\Psi_{n-3}-3(n-1)\sqrt{n-1}\Psi_{n-1}
\\
+3n\sqrt{n}\Psi_{n+1}+\tfrac{1}{3}\sqrt{n(n+1)(n+2})\Psi_{n+3}\Big).\end{multline}

\subsubsection{Second order}
We have that $w_1\varphi_{1}=O(\alpha_{1}^2)$, therefore we shall give priority to the term $w_2\varphi_{0}$ since $\alpha_{1}^2=o(\alpha_{2})$, see \eqref{b_bound2}. These considerations bring us to solve the equation 
$(h_0-\mu_{0})\varphi_{2}=(\mu_{2}-w_2)\varphi_{0}, $
and as above we get $\mu_{2}=\langle w_2\varphi_{0},\varphi_{0}\rangle$ and $\varphi_{2}=(h_0-E_{0})^{-1}(\mu_{2}-w_2)\varphi_{0}.$ Computations using \eqref{E:decomposetPsi} and \eqref{E:decomposet3Psi} provides
$$\langle t^4\Psi_{n},\Psi_{n} \rangle=\langle t^3\Psi_{n},t\Psi_{n} \rangle=\tfrac{3}{4}(2n^2-2n+1):=\gamma_{n},$$
implying  that $$\mu_{2}=\frac{3}{4}(2n^2-2n+1)\alpha_{2}.$$ Moreover, $\varphi_{2}$ has the form
$$\varphi_{2}=\alpha_{2}(k)\sum_{p=-2}^{2}c_{p}\Psi_{n+2p},$$
where, by construction $c_{0}=0$ and
\begin{equation}
\forall n\geq1, \, 
\left\{
\begin{aligned}
&c_{-2}=-\tfrac{1}{32}\sqrt{(n-1)(n-2)(n-3)(n-4)}\\
&c_{-1}=-\tfrac{1}{16}\sqrt{(n-1)(n-2)}(4n-6)\\
&c_{1}=\tfrac{1}{16}\sqrt{n(n+1)}(4n+2)\\
&c_{2}=\tfrac{1}{32}\sqrt{n(n+1)(n+2)(n+3)} \ .
\end{aligned}
\right.
\end{equation}

\subsection{ Application of the spectral theorem and proof of the asymptotics of the band function}

Denote by $(\mu,v^{\qm}_{n}(\cdot,k))$ the quasimodes constructed above according to the ansatz \eqref{E:anzatz}. Note that $v^{\qm}_{n}$ is obviously in the domain of $-\partial_{t}^2+w$. Then, by construction 
$$((-\partial_{t}^2+w)-\mu) v^{\qm}_{n}=w_1\varphi_{1}+w_1\varphi_{2}+w_2\varphi_{1}-\mu_{2}\varphi_{1}-\mu_{2}\varphi_{2}+{\omega_2\varphi_2}+R_{4}v^{\qm}_{n}.$$
Let $\eta_{n}(k)$ be the norm of the above quantity. Accordingly  
\begin{equation}
\label{E:defgamma}
\eta_{n}(k)=\|R_{4}v^{\qm}_{n}\|+O(\alpha_{1}^2)+O(\alpha_{1}\alpha_{2})=\|R_{4}v^{\qm}_{n}\|+O(\alpha_{1}^2).
\end{equation}
where we have used condition \eqref{b_bound2}.
 Recalling  that $R_{4}$ is the fourth order Taylor remainder of $w$, \Bk we get a constant $C>0$ such that
\begin{align*}
\forall k\in \R, \quad \| R_{4}v^{\qm}_{n}(\cdot,k) \|^2 & \leq C \int_{\R}\left(\sup_{|s|<t} |w^{(5)}(s,k) | \,   | t^5  |\right)^2v^{\qm}_{n}(t,k)^2 \rd t,
\end{align*}
with 
$$w^{(5)}(t,k)=2b_{k}^{-7/2}(10b'(x)b''(x)+5b(x)b^{(3)}(x)+(a(x)-k)b^{(4)}(x)).$$
 Let $\sigma_{0}\in (0,\frac{1}{(b_{+})^{1/2}})$, we consider the above integral in two pieces, $|t|<\sigma_{0} k$ and $|t|>\sigma_{0} k$. First we treat the part $|t|<\sigma_{0} k$, where  the potential $w^{(5)}$ is small: 
We have $|a(x)-k|\leq c |t|$ for all $(t,k)\in \R\times \R$ with $c>0$ a constant, see \eqref{E:CV}. Therefore, there exists a polynomial  $P_{n}$ (whose coefficients do not depend on $k$) such that
\begin{align*}
I_{1}:=&\int_{|t|<\sigma_{0} k}\left(\sup_{|s|<t} |\Bk w^{(5)}(s,k) |\Bk    |\Bk t  |\Bk    ^5 \right)^2v^{\qm}_{n}(t,k)^2 \rd t  
\\
 \leq& C \left(\sup_{|t|<\sigma_{0} k}( |\Bk b'b''   |\Bk  +  |\Bk  b^{(3)}   |\Bk  +   |\Bk   b^{(4)}   |\Bk   )(x) \right)^2 \int_{|t|<\sigma_{0} k}P_{n}(t)e^{-t^2} \rd t. 
 \end{align*}
Let $\sigma\in(0,\frac{1}{b_{+}}-\frac{\sigma_{0}}{(b_{+})^{1/2}})$. Recalling \eqref{E:CV}, we see that $|t| < \sigma_{0} k$ is equivalent to $|x-x_{k}|\leq\frac{\sigma_{0}}{b_{k}^{1/2}}k$. Using that $x_{k}\geq \frac{k}{b_{+}}$,
we obtain  $x> (\frac{1}{b_{+}}-\frac{\sigma_{0}}{b_{k}^{1/2}})k\geq\sigma k$ for $k$ large enough, since $b_{k} \to b_{+}$ as $k\to\infty$. Then, by assumption on $b$, we have $b^{(4)}=o(b^{(3)})$ and we find another constant $C>0$ such that for $k$ large enough
$$I_{1} \leq C \sup_{(\sigma k,+\infty)}( | b'b''    |+   |  b^{(3)}  |  )^2.$$
For the part $|t|>\sigma_{0} k$, we will use that the quasi-mode has exponential decay. By similar estimates as above, we obtain 
 \begin{align*}
 I_{2}:=&\int_{|t|>\sigma_{0} k}\left(\sup_{|s|<t}w^{(5)}(s,k)t^5 \right)^2v^{\qm}_{n}(t,k)^2 \rd t
 \\
  \leq &C\int_{|t|>\sigma_{0} k}P_{n}(t)e^{-t^2} \rd t \leq C( \tau ) e^{-\tau k^2},
 \end{align*}
 for all $\tau \in (0,\sigma_{0}^2)$. Therefore, using that $I_{1}+I_{2}=\| R_{4}v^{\qm}_{n} \|^2$, we get
\begin{equation}
\label{E:estimeR4}
\| R_{4}v^{\qm}_{n} \|= O(\sup_{({\sigma} k,+\infty)}(b'b''+b^{(3)})+e^{-\tau \frac{k^2}{2}}).
\end{equation} \Bk
Combining \eqref{E:estimeR4}  with \eqref{E:defgamma}, and recalling the definition of  $\epsilon$  in \eqref{E:defepsilon}, we can see that
\begin{equation}
\label{E:estimeeps}
\eta_{n}(k)=O(\epsilon(k)). 
\end{equation} 
Finally, noticing that $\|v^{\qm}_{n}\|=1+o(1)$   and applying  the spectral theorem
$$|E_{n}(k)-b_{k}\mu| =O(\epsilon(k)), $$
which together with \eqref{b_bound2} c) implies the asymptotics \eqref{E:asymptotEn}.

Denote by $q_{k}$ the quadratic form associated with $h(k)$. Each $E_{n}(k)$ is a simple eigenvalue of $h(k)$, and since each $E_{n}(k)$ converge to $\overline{{\mathcal E}_{n}}$ from  below, there exists $c>0$ and $k_{n}$ such that
$$\forall k\geq k_{n}, \ \ \dist(E_{n}(k),\sigma(h(k))\setminus\{E_{n}(k)\}) \geq c.$$ This spectral gap allows us to precise that the constructed quasi-modes are closed to the eigenfunctions of $h(k)$ in the following sense (see \cite[Proposition 2.5]{HellSj84}):
\Bk\begin{equation}
\label{E:spectralTbis}
\|u_{n}-u^{\qm}_{n}\|+\sqrt{q_{k}(u_{n}-u^{\qm}_{n})} =O(\epsilon(k)),
\end{equation}
where $u^{\qm}_{n}$ is the quasi-mode in the original variable, i.e.,
\begin{equation}
\label{E:defvn}
u^{\qm}_{n}(x,k):= b_{k}^{1/4}  v^{\qm}_{n}(b_{k}^{1/2}(x-x_{k}),k),
\end{equation}
 and $u_{n}$ is a normalized eigenfunction of $h(k)$ associated with $E_{n}(k)$.

\subsection{Asymptotics of the  derivatives}

The standard Feynman-Hellman (FH) formula writes 
\begin{align*}
E_{n}'(k)=&-2\int_{\R}(a(x)-k)u_{n}(x,k)^2 \rd x
\\
=&-2\int_{\R}(a(x)-k)u^{\qm}_{n}(x,k)^2 \rd x+2\int_{\R}(a(x)-k)(u_{n}(x,k)^2-u^{\qm}_{n}(x,k)^2)\rd x.
\end{align*}
The last term is easily controlled:
\begin{align*}
 & \left| \int_{\R}(a(x)-k)(u_{n}(x,k)^2-u^{\qm}_{n}(x,k)^2)\rd x \right|
 \\
 &\leq \left(\int_{\R}(a(x)-k)^2(u_{n}(x,k)-u^{\qm}_{n}(x,k))^2\rd x\int_{\R}(u_{n}(x,k)+u^{\qm}_{n}(x,k))^2\rd x\right)^{1/2}
 \\
 &\leq C q_{k}(u_{n}-u^{\qm}_{n})^{1/2}, 
\end{align*}
and is therefore $O(\epsilon(k))$, see \eqref{E:spectralTbis}. Now we have to compute the main term.
We do the same change of variable \eqref{E:CV} as above and we expand the potential:
$$a(tb_{k}^{-1/2}+x_{k})-k=b_{k}^{1/2}\left(t+\beta_{1}t^2+\beta_{2}t^3+\tilde{R}_{3}\right),$$
with $\beta_{1}=\frac{\alpha_{1}}{2}$ and $\beta_{2}=\frac{\alpha_{2}}{2}-\frac{\alpha_{1}^2}{8}$, see \eqref{E:defalpha}. Recall that the quasi mode in these variables writes $v^{\qm}_{n}=\varphi_{0}+\varphi_{1}+\varphi_{2}$, where $\varphi_{p}$ are defined in Subsection \ref{SS:QM}. Then we get
\begin{align*}
-2&\int_{\R}(a(x)-k){u}^{\qm}_{n}(x,k)^2 \rd x
\\
&=-2b_{k}^{1/2}\int_{\R}\left(t+\beta_{1}t^2+\beta_{2}t^3+\tilde{R}_{3}(t)\right)(\varphi_{0}(t)+\varphi_{1}(t)+\varphi_{2}(t))^2 \rd t.
\end{align*}
Expanding the above integrand, we use that some terms are odd functions and their integral cancels, obtaining 
\begin{equation}
\label{expandEpetreste}
E_{n}'(k)=-2b_{k}^{1/2}\left(T_{1}+T_{2}+T_{3}\right)+O(\epsilon_{n}(k)),
\ee
where 
$$T_{1}=\int_{\R}(t+\beta_{1}t^2)(\varphi_{0}(t)+\varphi_{1}(t))^2\rd t,$$
$$T_{2}=\int_{\R}2t\varphi_{1}\varphi_{2}+2\beta_{1}t^2\varphi_{0}\varphi_{2}+\beta_{1}t^2\varphi_{2}^2+2\beta_{2}t^3\varphi_{1}(\varphi_{0}+\varphi_{2})\rd t, \ \ \mbox{and} \ \ T_{3}=\int_{\R}\tilde{R}_{3}(t)v^{\qm}_{n}(t,k)^2\rd t.$$
The term $T_1$ satisfies 
$$T_{1}=2\langle t \varphi_{0},\varphi_{1}\rangle +\beta_{1}\|t \varphi_{0}\|^2+O(\alpha_{1}^2),$$
and it is known that $\|t \varphi_{0}\|^2=\frac{\Lambda_{n}}{2}$. More tedious compuations provides $\langle t \varphi_{0},\varphi_{1}\rangle=-\frac{3}{8}\alpha_{1}\Lambda_{n}$, so that 
$$T_{1}=-\tfrac{1}{2}\alpha_{1}\Lambda_{n}+O(\alpha_{1}^2).$$
Moreover,  using our assumptions on the magnetic field, $T_{2}=O(\alpha_{1}\alpha_{2})=O(b_{k}'b_{k}'')=o((b_{k}')^2).$ Finally, we estimate the term $T_{3}$ using  the same  arguments leading to \eqref{E:estimeR4}, which give us    $T_{3}=O(\sup_{(\widetilde{\sigma} k,+\infty)}b^{(3)}(k)+e^{-\tau k^2})$. Equation 
  \eqref{E:asymptotEnp} is obtained by combining these last estimates with \eqref{expandEpetreste}. 

Concerning the second derivative, we have the formula:
$$E_{n}''(k)=-2\left(-1+2\int_{\R}(a(x)-k)\partial_{k}u_{n}(x,k) u_{n}(x,k) \rd x\right).$$
Note that $\partial_{k}u_{n}$ satisfies the following ODE: 
$$(h(k)-E_{n}(k))\partial_{k}u_{n}(x,k)=(E_{n}'(k)+2(a(x)-k))u_{n}(x,k). $$
Due to the FH formula, the r.h.s. is orthogonal to $u_{n}$ and therefore we get 
$$\partial_{k}u_{n}(\cdot,k)=(h(k)-E_{n}(k))^{-1}((E_{n}'(k)+2(a(\cdot)-k))u_{n}(\cdot,k)).$$
We are left with the task of estimating 
\bel{13mar17a}\int_{\R}(a(x)-k)u_{n}(x,k) (h(k)-E_{n}(k))^{-1}((E_{n}'(k)+2(a(\cdot)-k))u_{n})(x) \rd x.\ee
We proceed in the same change of variables and Taylor expansion as above and we use resolvent formulas in order to apply $(h(k)-E_{n}(k))^{-1}$. We don't give all details here since these computations are rather tedious and redundant with those given above. First, arguing as before, we can replace $u_n$ by $u_n^{\rm qm}$ in formula \eqref{13mar17a}.  The factor $(a(x)-k){u}^{\qm}_{n}(x,k)$ is easily expressed  with Hermite's function as in the previous section. Further, we get 
 (the terms are ordered and grouped by decreasing order as $k\to+\infty$):
\medskip
\bel{E:expandencore}\ba{ll}
&(2(a(x)-k)+E_{n}'(k))u_{n}^{\qm}(x,k)
\\
=& 2b_{k}^{1/2}t\varphi_{0}+\left(E_{n}'\varphi_{0}+2b_{k}^{1/2}t\varphi_{1}+b_{k}^{1/2}\alpha_{1}t^2\varphi_{0}\right)+\left(2b_{k}^{1/2}t\varphi_{2}+2b_{k}^{1/2}\beta_{2}t^3\varphi_{0}\right)+R_{3}(t,k)
\\
:=&\tau_{0}+\tau_{1}+\tau_{2}+R_{3},
\ea\ee
with obvious notations, and where $R_{3}$ is a remainder. We want to apply the resolvent $(h(k)-E_{n}(k))^{-1}$ to the above expression.
Using the  resolvent  formula we obtain formally 
\begin{align*} &(h(k)-E_{n}(k))^{-1}
\\
&=b_{k}^{-1}(h_{0}-\Lambda_{n})^{-1}\left( \Id-( b_{k}^{-1}E_n(k)-\Lambda_n-(\alpha_{1}t^3+\alpha_{2}t^{4}+R_{4}))(h_{0}-\Lambda_{n})^{-1} \right)^{-1}
\\
&=b_{k}^{-1}(h_{0}-\Lambda_{n})^{-1}\left( \Id+( b_{k}^{-1}E_n(k)-\Lambda_n- (\alpha_{1}t^3+\alpha_{2}t^{4}+R_{4}))(h_{0}-\Lambda_{n})^{-1} +O(\alpha_{1}^2)\right)
\end{align*}
This formula is useful for applying the resolvent on $\tau_{0}$ given in \eqref{E:expandencore}. For the terms $\tau_{1}$ and $\tau_{2}$, which are respectively $O(\alpha_{1})$ and $O(\alpha_{2})$, we can apply directly $b_{k}^{-1}(h_{0}-\Lambda_{n})^{-1}$ instead of $(h(k)-E_{n}(k))^{-1}$ since the remainder will be $O(\epsilon)$.  

As above, we express everything in the basis of Hermite's functions, neglecting all the term that are controlled by $\alpha_{1}^2$. The above constructions insures that there is no $\varphi_{0}$ in each of these terms, so that it is possible to apply $(h_{0}-\Lambda_{n})^{-1}$. The same estimates as above lead to the asymptotics of $E_{n}''(k)$ announced in Theorem \ref{T:asymptband}. 
\Bk





\section{States localized near the thresholds}
\label{S:bulkstates}
 To begin we introduce some notations.  Let $\cU_{n}:L^2(\R)\mapsto L^2(\R\times \R)$ be the isometry defined by $(\cU_{n} w)(x,k):=w(k)u_{n}(x,k)$. Its adjoint is $(\cU_{n}^{*} v)(k)=\langle u_{n}(\cdot,k),v(\cdot,k) \rangle$, and we consider the projectors $\Pi_{n}:=\cU_{n}\cU_{n}^{*}$ and $\pi_{n}:=\cF^{*} \Pi_{n} \cF$. We know that $\sum_{n \geq1}\pi_{n}=\Id$. 
 
 Let $I$ be a Borel subset of $ \sigma(H_0)$. A function $\varphi\in \Dom(H_0)$ is localized in $I$ if it satisfies $\dP_{I}(H_0)\varphi=\varphi$, where $\dP_{I}(H_0)$ is the spectral projector  for $H_0$ associated with $I$. It is easy to see that $\varphi$ is localized in $I$ \Bk if and only if for all $n\in \N$, $\supp(\cU_{n}^{*}\cF \varphi)\subset E_{n}^{-1}(I)$. Assume that $I$ is bounded and that $I$ contains exactly one threshold $\overline{\mathcal E}_n$ (or that the distance from $I$ to $\overline{\mathcal E}_n$ is small). Then all $\pi_{m}\varphi$ with $m\neq n$ have standard {\it edge state} properties (see \cite{hissocc2}), whereas $\pi_{n}\varphi$ would be called the \emph{bulk} component of $\varphi$ in the terminology of \cite{dBiPu99,hispopsoc} (notice that the {\it bulk-edge} terminology is  a priori \Bk  irrelevant here since the system has no boundary).  The edge states are known to be functions localized in a bounded region (in the $x$ direction) of the plane, whose current (representing their transport properties in the $y$ direction) is bounded from below by a constant that depends only of the energy interval $I$, see \eqref{E:estimatecurrentrough} below. On the contrary, current of bulk states cannot be bounded from below, and they  are not localized in general.
 
 The appearance of bulk states is linked to one of the main original property of this model: the band functions are not proper (therefore this model does not enter the theory developed in \cite{GerNi98}). Indeed, if $I$ contains a threshold $\overline{\cE_{n}}$, the set $E_{n}^{-1}(I)$ is not bounded (or gets large if $I$ gets close to $\overline{\mathcal E}_n$). To treat the properties of quantum states localized in energy near a given threshold $\overline{\cE}_n$, it is necessary to have good understanding of the corresponding band function at infinity. In this section we will describe some properties of  the bulk component of  \Bk $\pi_{n}\varphi$, namely we will show that its current is small and that it is localized in the zone of configuration space where $H_0$ is closed to the Landau Hamiltonian with constant magnetic field $b_{+}$.

\subsection{Estimates on the current}
The current operator is $J_{y}:=-i[H_0,y]$, defined as a self adjoint operator on $\Dom(H_0)$. Easy computations show that $J_{y}=-i\partial_{y}-a$. We introduce some notations in order to link $J_{y}$ with the derivative of the band functions. 

The Feynman-Hellmann formula shows that the velocity operator satisfies
 $$\langle J_{y} \pi_{n} \varphi,\pi_{n} \varphi \rangle = \langle E_{n}'\cU_{n}^{*}\cF \varphi,\cU_{n}^{*}\cF \varphi  \rangle, $$
 see \cite{manpu}.
Notice that this form is reminiscent of \cite[Section 5.2]{Yaf08}, stating that in a Fourier basis adapted to the  diagonalization  of $H_{0}$, the velocity operator is unitarily equivalent to the multiplication by the derivative of the band functions.

Let $I\subset \sigma(H_0)$ and let $\varphi$ be such that $\dP_{I}(H_0)\varphi=\varphi$. Then the above identities yield
\begin{equation}
\label{E:estimatecurrentrough}
 \inf_{E_{n}^{-1}(I)}E_{n}'  \leq \frac{\langle J_{y} \pi_{n}\varphi,\pi_{n}\varphi \rangle }{\|\pi_{n}\varphi\|^2}\leq \sup_{E_{n}^{-1}(I)}E_{n}'. 
 \end{equation}
Therefore, as explained in \cite{manpu,hissocc2}, a state $\pi_n\varphi$ localized in energy away from  the  thresholds $\overline{\mathcal{E}}_n$  and $\underline{\mathcal{E}}_n$    has a positive current, whereas when $I$ contains  such a threshold,  no lower bound on the current is available.

Define $\varrho_n:(0,\overline{\mathcal E}_n-\underline{\mathcal E}_n)\to \R$,  as the inverse function of $\overline{\mathcal E}_n-E_n$.  Under the assumption \eqref{E:HypB}, we deduce from Corollary \ref{C:asymptBexplicit} the following asymptotics, as $\delta\to0$:
\bel{E:asymptvarrho}
\varrho_{n}(\delta)=\alpha\delta^{-\frac{1}{M}}+O(\delta^{\frac{1}{M}}) \ \ \mbox{and} \ \ \varrho_{n}'(\delta)=-\beta\delta^{-\frac{1}{M}-1}+O(\delta^{\frac{\eta}{M}}),
\ee
 with $\alpha=b_{+}\Lambda_{n}^{\frac{1}{M}}$ and $\beta=\frac{b_{+} \Lambda_{n}^{\frac{1}{M}}}{M}$, and $\eta$  defined in Corollary \ref{C:asymptBexplicit}. 

Following the approach introduced in \cite{hispopsoc},  we consider an energy interval $I_{b}:=(\overline{\mathcal E}_n-\delta_{1},\overline{\mathcal E}_n-\delta_{2})$ with $0<\delta_{1}<\delta_{2}$, and  provide quantitative upper and lower bound for the current of states localized in energy in $I_{b}$, as $\delta_{i}\to0$.

\begin{proposition}\label{corocurrent}
Assume \eqref{b_bound} and \eqref{E:HypB}. Then there exists $\eta>0$ such that for all $\varphi\in D(H_0)$ satisfying $\dP_{I_{b}}(H_0)\varphi=\varphi$, there holds, as $\delta_{i}\to 0$: 
$$\frac{\delta_{1}^{1+\frac{1}{M}}}{\beta}+O(\delta_{1}^{\frac{\eta}{M}})\leq\frac{ \langle J_{y} \pi_{n}\varphi,\pi_{n}\varphi \rangle}{\| \pi_{n}\varphi \|^2} \leq \frac{\delta_{2}^{1+\frac{1}{M}}}{\beta}+O(\delta_{2}^{\frac{\eta}{M}}).$$ Here the remainders are uniform with respect to $\delta_{i}$ and $\varphi$. 
\end{proposition}
\begin{proof}
Note that $(E_{n}'\circ E_{n}^{-1})(y)=-\frac{1}{\varrho_{n}'(\overline{\mathcal E}_n-y)}$. Let $I\subset \sigma(H_0)$, assume that $\cT_{H_0}\cap I=\emptyset$, and define the set  $\tau(I):=-I+\overline{\mathcal E}_n$. Let $\varphi$ be such that $\dP_{I}(\varphi)=\varphi$.  Then, \eqref{E:estimatecurrentrough} provides directly
 \bel{estimatecurrentgeneral}
 \inf_{y\in\tau(I)}\frac{-1}{\varrho_{n}'(y)} \leq \frac{\langle J_{y} \pi_{n}\varphi,\pi_{n}\varphi \rangle }{\|\pi_{n}\varphi\|^2} \leq  \sup_{y\in\tau(I)}\frac{-1}{\varrho_{n}'(y)}. \ee
We now apply this estimate with $I=I_{b}$ and we conclude using \eqref{E:asymptvarrho}.
\end{proof}
\begin{remark} Formula \eqref{estimatecurrentgeneral} is very general and gives estimates on the current depending on intrinsic quantities. Using Theorem \ref{T:asymptband},  we  could   show under more general conditions that $\varrho_{n}'$ behaves roughly as $\delta\mapsto b_+/(\Lambda_nb'(\frac{\varrho_n(\delta)}{b_+}))$, as $\delta \downarrow 0$, and we  could give the  corresponding version of Proposition \ref{corocurrent}.  However, to keep the statements simple, we preferred to consider more explicit magnetic fields. 
\end{remark}
\Bk

\subsection{(De)localization of states  at energy near thresholds \Bk }

It is known that states whose energy is far from thresholds are localized in $x$-direction, in the sense that they have exponential decay at infinity, and this decay can be estimated using the distance from their energy to the threshold, see \cite{hissocc2}. Such states are usually called {\it edge states}, because they typically arise in quantum systems with a boundary, although  it is known  that \Bk open magnetic system such as the one considered here can exhibit  this kind of behavior (\cite{ExJoyKov99}). On the contrary, states whose energy is close to the thresholds are called {\it bulk states} since they are usually located far from the boundary, if this exists.  Here we give quantitative estimates on states, as their energy tends to the limit of a band function, showing that they are located  in a   region where  the variation of the magnetic field is small.
 \begin{theorem}\label{theo_loca}
Set $I_{b}(\delta):=(\overline{\mathcal E}_n-\delta,\overline{\mathcal E}_n)$, { and assume \eqref{b_bound} and \eqref{E:HypB}.
Then, there exists  positive constants 
$c$,
 $C$ and $\delta_0$}, such that $\forall \delta\in (0,\delta_{0})$\rm{:}
\bel{E:bulklocali}
\forall \varphi\in Ran(\dP_{I_{b}(\delta)}(H_0)), \quad \int_{-\infty}^{ c \delta^{-\frac{1}{M}} }\int_\R|\pi_{n}\varphi(x,y)|^2 \rd y \,\rd x  \leq C { \delta^{ \frac{\eta}{M}\Bk}\|\pi_{n}\varphi\|^2},
\Bk
\ee
where $\eta$ has been defined in Corollary \ref{C:asymptBexplicit}.
 \end{theorem}
\begin{proof}
 { Consider the function $\epsilon$  defined in \eqref{E:defepsilon}, and  set $r_{n}(\delta):=\sup_{k\in \varrho_{n}((0,\delta))} \epsilon(k)$ and $k_{-}(\delta):=\inf \varrho_{n}((0,\delta))=\varrho_{n}(\delta).$ For $\nu>1$, let 
$$x(\delta):=x_{k_{-}(\delta)}-\frac{\nu}{b^{+}}\sqrt{|\log r_{n}(\delta)|}.$$ 
 We will prove the more precise result: 
 \bel{E:localizationsharp}
 \forall \varphi\in Ran(\dP_{I_{b}(\delta)}), \quad \int_{-\infty}^{x(\delta)}\|\pi_{n}\varphi(x,\cdot)\|_{L^2(\R)}^2 \rd x  \leq Cr_{n}(\delta)\|\pi_{n}\varphi\|^2.
 \ee
 From Corollary \ref{C:asymptBexplicit} one can easily prove that there exist  $A, B>0$ such that
 \bel{jan12_17}\ba{l}
 x_{k_{-}(\delta)}=A\delta^{-1/M}(1+o(1))\\
 r_n(\delta)=B \delta^{ \frac{ \eta}{M} }(1+o(1)), \quad  \delta \downarrow 0,
 \ea\ee} and therefore $x(\delta) \geq c \delta^{-\frac{1}{M}}$ for $\delta$ small enough with $c>0$ a constant.  Thus, \eqref{E:bulklocali} is implied using \eqref{E:localizationsharp} and \eqref{jan12_17}.\Bk
 
 Now, let us prove \eqref{E:localizationsharp}. We use that $\|\pi_{n}\varphi(x,\cdot)\|_{L^2(\R)}=\|\varphi_{n}u_{n}(x,\cdot)\|_{L^2(\R)},$
 where $\varphi_{n}:=\cU_{n}^{*}\cF \varphi$. Moreover, $\dP_{I_{b}(\delta)}(H_0)\varphi=\varphi$ implies that $\supp(\varphi_{n})\subset \varrho_{n}((0,\delta))$, and therefore
  \begin{align}\begin{split}
    \int_{-\infty}^{x(\delta)} \|\pi_{n}\varphi(x,\cdot)\|_{L^2(\R)}^2 \rd x &=\int_{\varrho_{n}((0,\delta))}\varphi_{n}(k)^2 \int_{-\infty}^{x(\delta)}u_{n}(x,k)^2 \rd x \rd k
  \\
\label{E:decomposphsp}  &= \int_{\varrho_{n}((0,\delta))}\varphi_{n}(k)^2 \int_{-\infty}^{x(\delta)}u^{\qm}_{n}(x,k)^2 \rd x \rd k+O(\epsilon(k)),
\end{split}  \end{align}
  where $u^{\qm}_{n}$ is the quasimode constructed in Section \ref{SS:QM}, see \eqref{E:defvn}. 
Thus, we are left with the task of estimating $\int_{-\infty}^{x(\delta)}u^{\qm}_{n}(x,k)^2\rd x$ for $k\in \varrho_{n}((0,\delta))$. To this end,  we use the same change of variables \eqref{E:CV}, and noticing that $x(\delta)\to+\infty$ as $\delta\to0$ we found that 
$$\int_{-\infty}^{x(\delta)}u^{\qm}_{n}(x,k)^2 \rd x=\int_{-\infty}^{b_{k}^{1/2}(x(\delta)-x_{k})} v^{\qm}_{n}(t,k)^2 b_{k}^{-1/2} \rd t=(1+O(\alpha_{1}(k))b_{k}^{-1/2}\int_{-\infty}^{b_{k}^{1/2}(x(\delta)-x_{k})} \Psi_{n}(t)^2  \rd t.$$
Since $a^{-1}$ is increasing, by definition of $k^{-}$, $x_{k^{-}}-x_{k}\leq 0$ for all $k\in \varrho_{n}((0,\delta))$. Then, 
$$b_{k}^{1/2}(x(\delta)-x_{k}) \leq b_{k}^{1/2}(x(\delta)-x_{k^{-}})=-\nu \left(\tfrac{b_{k}}{b_{+}}\right)^{1/2}\sqrt{|\log r_{n}(\delta)|} <0.$$
Moreover, as $L\to-\infty$ , there holds $$\int_{t \leq L} \Psi_{n}(t)^2 \rd t=O(L^{2n+1}e^{-L^2}).$$
Accordingly, for a constant $C>0$ and  $\delta$  small:
$$\forall k\in \varrho_{n}((0,\delta)), \quad \int_{-\infty}^{x(\delta)}u^{\qm}_{n}(x,k)^2 \rd x \leq C \sqrt{|\log r_{n}(\delta)|}^{2n+1}r_{n}(\delta)^{\nu^2\left(\tfrac{b_{k}}{b_{+}}\right)}.$$
Finally, for $\delta$ small enough, $\nu^2\left(\tfrac{b_{k}}{b_{+}}\right)> 1$ \Bk for all $k\in \varrho_{n}((0,\delta))$, which implies 
$$\forall k\in \varrho_{n}((0,\delta)), \quad \int_{-\infty}^{x(\delta)}u^{\qm}_{n}(x,k)^2 \rd x \leq C r_{n}(\delta).$$
We get \eqref{E:localizationsharp} by combining the last inequality  with \eqref{E:decomposphsp}  and  $\int_{k}\varphi_{n}(k)^2 \rd k=\|\pi_{n}\varphi\|^2$.
\end{proof}
 
\begin{remark}
Using the more precise (and more general) estimate \eqref{E:localizationsharp}, it is possible to get the best possible constant $c$ in \eqref{E:bulklocali}, up to a correction term in $\sqrt{\log \delta}$. We have preferred to state \eqref{E:bulklocali} for readability. 
\end{remark}
\Bk
\section{The spectral shift function}\label{secSSF}
Consider an electric potential $V : \R^2 \to [0,\infty)$, that is a Lebesgue measurable {function} satisfying 
    \bel{cd2}
    V(x,y)\leq C \langle x,y \rangle  ^{-m}, \quad (x,y) \in \R^2,
\ee for  some  positive constant $C$, $m > 2$, and where $\langle x, y \rangle:=(1+x^2+y^2)^{1/2}$. On the domain of $H_0$ introduce the operator 
$$H: = H_0 + V,$$ self-adjoint in $L^2(\R^2)$.
Estimate \eqref{cd2} combined  {with} the diamagnetic inequality  imply that for any real $E_0 < \inf \sigma(H)$ the operator
 $V^{1/2} (H_0 - E_0)^{-1}$ is Hilbert--Schmidt, and hence the resolvent difference $(H-E_0)^{-1}-(H_0-E_0)^{-1}$ is
a trace-class operator. This last property implies that 
there exists a unique function $\xi = \xi(\cdot; H, H_0) \in L^1(\R; (1+E^2)^{-1}\rd E),$ called the
Spectral Shift Function (SSF) for the operator pair $(H,H_0)$, that satisfies the Lifshits-Kre\u{\i}n trace formula:
$$\mbox{Tr}(f(H)-f(H_0))=\int_\R \xi(E;H,H_0) f'(E)\rd E, $$
for each $f \in C_0^\infty(\R)$, and vanishes identically in $(-\infty, \inf \sigma(H))$ \cite{yaf}. 

The SSF can be seen as  the scattering phase of the operator pair $(H,H_0)$, namely  we have the  Birman-Kre\u{\i}n formula:
\bel{scatteringSSF}
{\rm det}(S(E))=e^{-2\pi i\xi(E)}, \quad E \in \sigma_{ac}(H_0) \, {\rm a.e.}, 
\ee
where $S(E)$ is the scattering matrix of the operator pair $(H,H_0)$.

In addition, we have the following relation between the SSF  and the eigenvalue counting function. \Bk Condition \eqref{cd2} implies that the essential spectrum of $H$ is given by 
$\bigcup_{n\geq 1}[\underline{\mathcal E}_n,\overline{\mathcal E}_n]$. Then,  if we  suppose that there is a finite gap in this set, let say  $
\overline{\mathcal E}_n<\underline{\mathcal E}_{n+1}$, we may define:$$
{\mathcal N_n}(\lambda)=\rm{Rank}\,\dP_{({{\mathcal E}}_n+\lambda,{{\mathcal E}}_{n+1})}(H),$$
which is the function that counts the number of discrete eigenvalues of $H$ on the interval $(\overline{\mathcal E}_n+\lambda,\underline{\mathcal E}_{n+1})$. 
From \eqref{pssf} below and the Birman-Schwinger principle we can see that 
\bel{15_sep_16}    {\mathcal N}_n(\lambda) =\xi({{\mathcal E}}_n+\lambda;H_{0}+V,H_0)+O(1),\quad \lambda\in(0,\underline{\mathcal E}_{n+1}-\overline{\mathcal E}_n)\ee
i.e., outside the essential spectrum both functions agree up to a bounded term, and therefore  the SSF could be seen as an extension of 
$ {\mathcal N}_n$ to  the whole real line\footnote[2]{Although we have to be aware that the SSF is only defined as an element of $L^1(\R; (1+E^2)^{-1}\rd E).$}.  As we already mentioned in the introduction, there exist considerable works that study the distribution of the discrete eigenvalues of $H$ by giving the asymptotic behavior of ${\mathcal N_n}(\lambda)$, particularly for the constant magnetic field case. For $b$ non-constant there exist some works as well, but to the best knowledge of the authors, the spectral density inside the continuous spectrum has not been considered yet. \Bk  

In this article, for the class defined by the SSF in  $L^1(\R; (1+E^2)^{-1}\rd E),$ we are going to take Pushnitski's  representative given by \eqref{pssf} below, and our main theorems will refer to it.


Our goal is to study the behavior of the SSF  when a thresholds is approached by below or by above, for negative or positive perturbations. In what follows, we have stated our results for thresholds $\overline{\cE_{n}}$, corresponding to the limit in $+\infty$ of the band functions, in link with the fact that $b(x)\to b^{+}$ as $x\to+\infty$. If there exists $p\in \N$ such that $\overline{\cE_{n}}=\underline{\cE_{p}}$,  the results are more difficult to read, because the possible singularity of the SSF at $\overline{\cE_{n}}$ will not come only from the $n$-th band function, but also from the fact that $\lim_{k\to-\infty}E_{p}(k)=\overline{\cE_{n}}$ . For that reason, from now on we will assume 
\bel{IsolatedT}
\forall p\in \N, \quad \overline{\cE_{n}}\neq\underline{\cE_{p}},
\ee 
which is equivalent to say that   $\frac{b_{-}}{b_{+}}$ is not a ratio of odd integers. 

The following theorem shows that for some non-constant magnetic fields and  suitable conditions on $V$,  the SSF ``extends" the    properties of continuity and boundedness of the eigenvalue counting function into the continuous spectrum of $H$.

\begin{theorem}\label{thssf}
Suppose that  $V\geq0$  satisfies \eqref{cd2} and  that $b$  satisfies \eqref{b_bound} and \eqref{E:HypB}. 
Set $H_\pm:=H_0\pm V$, then:
\begin{enumerate}
\item\label{thssf1a}  On any compact set $\mathcal{C}\subset \R\setminus\cT_{H_0}$, 
$\sup_{E\in{\mathcal C}}\xi(E;H_\pm,H_0)<\infty,$
i.e.  the SSF is bounded  away from the thresholds. Moreover, 
as $\lambda \downarrow 0$:
\bel{SSFcontrolthresh}
 \xi(\overline{\mathcal{E}}_n-\lambda;H_\pm,H_0)=O(\lambda^{-1}+\lambda^{-2/M}),
\quad \xi(\overline{{\mathcal E}}_n+\lambda;H_{+},H_0)=O(\lambda^{-1}), \ee
\bel{SSFboundedthresh}
\xi(\overline{{\mathcal E}}_n+\lambda;H_{-},H_0)=O(1).
\ee
\item\label{thssf1} The SSF $\xi(\cdot ;H_\pm,H_0)$ is continuous on $\R \setminus \big(\sigma_p(H_\pm)\cup \cT_{H_0}\big)$, where $\sigma_p(H_\pm)$ denote the set of eigenvalues of $H_\pm$. 
\end{enumerate}
\end{theorem}

 \begin{remark} 
 It is obvious that under conditions of Theorem \ref{thssf}, the function ${\mathcal N_n}(\lambda)$ satisfies 
 the same boundedness and continuity properties, wherever it  is defined. Similar results for different magnetic Hamiltonians have been obtained before, see for example \cite{bpr, bsoccrai, brumi1}.
\end{remark}
From  Theorem  \ref{thssf} we know that the only possible  points for $\xi(E;H_\pm,H_0)$  to be  unbounded,  are the thresholds in $\cT_{H_0}$. 
 In the following two Theorems we show that this also depends on the relation between the decaying rate of $V$ and and the convergence rate of  $b(x)$ to $b_+$. \Bk 
Furthermore, in the unbounded case we will obtain the explicit  asymptotic behavior  of $\xi$ at the thresholds. 

\begin{theorem}\label{thVcompact}
 Assume that   $V\geq0$ satisfies \eqref{cd2}. Suppose also that $b$  satisfies \eqref{b_bound} and  \eqref{E:HypB}  with $m>M+2$.  Then,  as $\lambda\downarrow 0$,  $$\xi(\overline{{\mathcal E}}_n+\lambda;H_{\pm},H_0)=O(1) \ \ \mbox{and} \ \ \xi(\overline{{\mathcal E}}_n-\lambda;H_{\pm},H_0)=O(1).$$
\end{theorem}
 This theorem covers four different cases. One of these was already stated, under weaker asumption on $V$, in \eqref{SSFboundedthresh}, and is conceptualy very different. 
 For the proof of this theorem is not necessary that  condition \eqref{cd2} holds isotropically. It is enough if for some $x_0\in \R$ it holds for $x>x_0$  with $m>M+2$, and for $x<x_0$ with $m>2$.  Notice that the conditions of the Theorem include the case where $V$ is compactly supported.

 To obtain the asymptotic behavior of $\xi$ in the unbounded case, we need to impose more restrictive conditions on $V$. First, \Bk we will assume that  for any pair $(\alpha, \beta) \in \mathbb{Z}_+^2$ ($\mathbb{Z}_+:=\{0,1,2,...\}$),  there exists a positive constant $C_{\alpha,\beta}$ satisfying 
\bel{apr14}
|\partial_x^\beta\partial_y^\alpha V(x,y)|\leq C_{\alpha, \beta}\langle x, y \rangle^{-m-\alpha-\beta} \quad \mbox{ for all} \,\,(x,y)\in \R^2,
\ee
where $m>2$. 

Next, let  $f:\R^2\to\R$ be a measurable function and  
for $\lambda >0$ set 
$$
N_0(\lambda,f):=\frac{1}{2\pi}vol\{(x,y)\in \R^2; f(x,y)> \lambda, x>0\},
$$
where $vol$ denotes  the Lebesgue measure in $\R^2$.

 Then, we will suppose  that  for some positive constants $C$ and  $\lambda_0$
\bel{jul10a}
N_0(\lambda,V)\geq C \lambda^{-2/m}, \quad 0<\lambda<\lambda_0,
\ee
and  that  $N_0(\lambda,V)$  satisfies a homogeneity  condition of the form
\bel{jul10}
\lim_{\epsilon \downarrow 0}\limsup_{\lambda\downarrow 0}\,\lambda^{2/m}\left(N_0(\lambda(1-\epsilon),V)-N_0(\lambda(1+\epsilon),V)\right)=0.
\ee

 Conditions \eqref{apr14}, \eqref{jul10a} are necessary to have some control of the function $V$ at infinity, from above and below (Note that if $V$ satisfies  \eqref{apr14}, then $N_0(\lambda,V) = O(\lambda^{-\frac{2}{m}})$,  as $\lambda\to0$). Meanwhile,  \eqref{jul10} is a regularity condition. They are usually assumed in the study of this kind of asymptotics (see for example \cite{dauro, rai1, IwaTam98, shi2}). \Bk {The existence of a limit of the form 
$$\lim_{(x,y)\to \infty} \langle x, y\rangle^{m}V(x,y)=\omega\left(\frac{(x,y)}{|(x,y)|}\right),$$   where  $\omega: S^1\to [\epsilon, \infty)$ is smooth and  $\epsilon>0$, guarantees that  $V$   satisfy \eqref{apr14}, \eqref{jul10a} and \eqref{jul10}.
 


\begin{theorem}\label{thssf2}
Assume that   $V\geq0$ satisfies \eqref{apr14} with $m>2$ and   $N_0(\lambda, V)$ satisfies  \eqref{jul10a}, \eqref{jul10}. Suppose also that $b$  satisfies \eqref{b_bound} and  \eqref{E:HypB} \ with $M>m$. Then,  the following asymptotic formula  at the thresholds holds true: 
\bel{asym_ssf}
\xi(\overline{{\mathcal E}}_n\pm\lambda;H_\pm,H_0)=\pm b_+\,{N_0(\lambda,V)}(1+o(1)), \quad \lambda \downarrow 0.
\ee

\end{theorem}

\begin{remark} Equation  \eqref{asym_ssf} is similar to the results obtained in \cite{brumi1} for  the SSF of some magnetic Schr\"odinger operators defined in a Half-plane, and in \cite{rai1}, \cite{Mir16} for the eigenvalue counting function of the perturbed Landau Hamiltonian and Iwatsuka Hamiltonian, respectively.

Moreover,  \eqref{asym_ssf} could be written in the form
\bel{iwa_tam}\xi(\overline{{\mathcal E}}_n\pm\lambda;H_\pm,H_0)
=\pm\int\displaylimits_{\displaystyle{\left\{(x,y)\in \R^2;\, V(x,y)>\lambda, \,x>0\right\}}}b(x)\,\rd x \rd y \,(1+o(1)), \quad \lambda \downarrow 0,\ee
which is in accordance with the formula obtained for the function that counts number of  discrete  eigenvalues  of the  perturbed Pauli operator near zero (see \cite{IwaTam98}). This formula 
 shows that the asymptotic behavior of $\xi$ depends on the whole magnetic field, but if $b(x)$ converges to $b_+$ fast enough it is only necessary to consider the limit $b_+$. On the contrary, if the convergence of $b(x)$ to $b_+$ is slow in comparison with the decaying rate of $V$, the formula \eqref{iwa_tam} is no longer true, as is shown by  Theorem \ref{thVcompact}, because the term in the right hand side of \eqref{iwa_tam} is unbounded  as $\lambda\downarrow 0$, for example when $V(x,y)=\langle x, y \rangle^{-m}$ \Bk and $m>M+2$.  Then Theorem \ref{thVcompact} is somehow a non-semiclassical result. \end{remark}

\begin{remark} Using the results of \cite{hispopsoc} combined  with the methods of  \cite{brumi1},  it should be possible to prove \eqref{asym_ssf} for a   magnetic field satisfying $b(x)=b_+$ for $x>0$, and  $b(x)=b_-$  for  $x<0$. Note that this magnetic field  does not satisfy \eqref{E:HypB} nor \eqref{b_bound2}. 
\end{remark}

The proofs of Theorems \ref{thssf},\, \ref{thVcompact} and \,\ref{thssf2} appear  in the following two sections. They are inspired by \cite{brumi1} and \cite{bsoccrai}.

\section{A preliminary limiting absorption principle and other results } \label{ss21} 
 
Let $(\mathfrak{S}_p,\|\cdot\|_p)$, $p\geq1$, be the Schatten -- von Neumann class of compact operators. Define the operator valued function $G_n:\R\to \mathfrak{S}_2(L^2(\R^2);\C)$, by
$$G_n(k)f:=\frac{1}{\sqrt{2\pi}} \int_{\R}\int_\R 
e^{-iky}V^{1/2}(x,y)u_n(x,k)f(x,y)\,\rd x\rd y, \quad f \in L^2(\R^2),$$
the function $u_n$ being the normalized eigenfunction of $h(k)$ as in section \ref{sec_asymp_band}. 

\begin{lemma}\label{boch_lipt}For any $n\in\N$:
\begin{enumerate}
\item The  function $\|G_n^*G_n\|$ is in  $L^1(\R),$ 
\item The operator function $G_n^*G_n$ satisfies a Liptschitz condition at infinity:  There exist $C>0$ and $k_{0}\in \R$ such that for all $k$ and $k'$ in $(k_{0},\infty)$
$$ \|G_n^*G_n(k)-G_n^*G_n(k')\| \leq C |k-k'|.$$
\end{enumerate}
\end{lemma}
\begin{proof} Since $G_n^*G_n(k)$ is of finite rank, we can use the different norms indistinctly. The first statement follows using  that   for some $\tau>0$ \bel{3mar17c}\sup_{(x,k)\in \R^2}|u_n(x,k)e^{\tau(x-x_{k})^2}|<\infty,\ee  \cite[Lemma 4.7]{shi2}, and \eqref{cd2}, which imply  
$$\ba{ll}
||G_{n}(k)^*G_{n}(k)||^2_{2}
&=\displaystyle{\frac{1}{2\pi}\int_{\R^2}\int_{\R^2} V(x,y)V(x',y')\; |u_n(x,k)u_n(x',k)|^2\,\rd x\rd x'\,\rd y \rd y'}\\[1em]
&\leq\displaystyle{\frac{1}{2\pi}\left(\int_\R\langle y\rangle
^{-m/2}\,dy\right)^2\left(\int_{\R}\langle
x\rangle^{-m/2}u_n(x,k)^2\,\rd x\right)^2}
\\[1em]
&\leq C\left(\displaystyle{ \int_{\R} \langle
x\rangle^{-m/2}e^{-2\tau(x-x_{k})^2}\rd x}\right)^2.\\
\ea$$
The last integral  is in $L^1(\R_k)$ by Young's inequality and $m>2$.


For the second statement, recalling the definition of $\Pi_n$ at the beginning of section \ref{S:bulkstates}, we can see that
$$\ba{ll}
&\|G_n(k)^*G_n(k)-G_n(k')^*G_n(k')\|_{2}\\[.2em]
\leq&\displaystyle{\frac{1}{2\pi}\left(\sup_{x \in \R}\int_\R V(x,y)\,\rd y\right)\|\Pi_n(k)-\Pi_{n}(k')\|_{2}}\\[1em]
\leq &\displaystyle{C\oint_\gamma} \|(h(k)-\omega)^{-1}-(h(k')-\omega)^{-1}\|\,d|\omega|\\[.5em]
\leq&\displaystyle{ C |k-k'|\oint_\gamma} \left\|(h(k)-\omega)^{-1}(a(x)-k')(h(k')-\omega)^{-1}\right\|d|\omega|,
\ea$$
 where the contour  $\gamma$ is such that the only eigenvalues of $h(k)$ and $h(k')$ inside it are $E_{n}(k)$ and $E_{n}(k')$. In order to find a uniform bound of   $\oint_\gamma \left\|(h(k)-\omega)^{-1}(a(x)-k')(h(k')-\omega)^{-1}\right\|d|\omega|$ for $k,k'$ big,
  it is enough to use the  inequalities
$$
\begin{array}{ll}&\|(a(x)-k)(h(k)-\omega)^{-1}g\|_{L^2(\R)}\\
\leq &\|(h(k)-\omega)^{-1}g\|_{L^2(\R)}+\|(a(x)-k)^2(h(k)-\omega)^{-1}g\|_{L^2(\R)}\\
\leq &\|(h(k)-\omega)^{-1}g\|_{L^2(\R)}+
C \|(h(k)+2b_+)(h(k)-\omega)^{-1}g\|_{L^2(\R)},
\end{array}$$
valid for any $g\in L^2(\R)$. In the last inequality we have used \cite[Theorem 1]{evgi}.\end{proof}
By Lemma \ref{boch_lipt}, for any $z\in \C_+$ and $l\in \N$,  it is posible to define
the integral 
$\int_\R\frac{G_l(k)^*G_l(k)}{E_l(k)-z}\,\rd k,$
which is a trace class operator. On the other side, we can define 
$$ T(z):= V^{1/2} (H_0-z)^{-1} V^{1/2},$$
and, from \eqref{direc} it is easy to see that
\bel{16mar17a}T(z)=\sum_{l\in\N}\int_\R\frac{G_l(k)^*G_l(k)}{E_{ l}(k)-z}\,\rd k.\ee
 We will show that $T(z)$ admits a limit as  $z\downarrow E$, for $E\in\Bk\sigma(H_{0})\setminus \cT_{H_0}$. This limit  is  linked to the SSF through the Pushnitski representation \eqref{pssf} given below.

Recall that  $\varrho_l:(0,\overline{\mathcal E}_l-\underline{\mathcal E}_l]\to \R$ is defined  as the inverse function of $\overline{\mathcal E}_l-E_l$. Therefore,   for any $l \in \N$:
$$\forall z\in \C\setminus\sigma(H_{0}), \quad \int_{\R}\frac{G_l(k)^*G_{l}(k)}{E_{l}(k)-z}\,\rd k=\int_{0}^{\overline{\mathcal E}_l-\underline{\mathcal E}_l}\frac{G_l(\varrho_{l}(s))^*G_l(\varrho_{l}(s))}{s+z-\overline{\cE}_{l}}\varrho_{l}'(s) \rd s.$$
\begin{proposition}\label{pr1}
Assume that $V$ satisfies \eqref{cd2}. Then, for $E\in\R\setminus\cT_{H_0} $ the limit 
$$
\lim_{\delta\downarrow 0}T(E+i\delta) = T(E+i0)
$$
exists in the norm sense.  Moreover:
\begin{enumerate}
\item[I.] The operator function ${\rm Im}\,T(\cdot+i0)$  is continuous in the trace class norm in $\R\setminus\cT_{H_0}$. Further, the rank of ${\rm Im}\,T(E+i0)$ is finite and  constant between two consecutive  thresholds.
\item[II.] The real part ${\rm Re}\,T(\cdot+i0)$   is continuous in the operator norm in $\R\setminus\cT_{H_0}$. 
\end{enumerate}
\end{proposition}

\begin{proof}

Define  $\gamma_n^-$ as  the nearest threshold below $\overline{{\mathcal E}}_n$ i.e.,
${{ \gamma}}_n^-:=\max\{\gamma\in \cT_{H_0}, \gamma <\overline{{\mathcal E}}_n\}$,
and consider  the set 
\bel{D:Ln}
L_n:=\{l\in\N;(\gamma_n^-,\overline{{\mathcal E}}_n) \cap{E_l(\R)}\neq\emptyset\}.
\ee
 Note that, for instance,  $L_{n}=\{n\}$, when $\gamma^-_{n}=\underline{\cE_{n}}$  and  $L_{n}=\{n,n+1\}$, when $\gamma^-_{n}=\underline{\cE_{n+1}}$. 

Seting $\gamma_1^-:=-\infty$, we have $\R\setminus\mathcal{T}_H=\bigcup_{n\in\N}(\gamma_n^-,\overline{\cE}_{n})\cup(\overline{\cE_{n}},\gamma_{n+1}^-)$. We will show the  proof of  the proposition on the set $(\gamma_n^-,\overline{\cE}_{n})$. The other intervals  can be treated  similarly.

Define the projection $P_n:=\sum_{l\not\in L_n}\pi_l$, where $\pi_l$ was defined at the beginning of section \ref{S:bulkstates}. Then, 
$$\sum_{l\not\in L_n}\int_\R\frac{G_l(k)^*G_l(k)}{E_l(k)-z}\,\rd k=V^{1/2}P_n(H_0-z)^{-1}P_nV^{1/2}$$
 depends analytically on  $z \in \C^{+}\cup (\gamma_n^-,\overline{{\mathcal E}}_n)$, 
  in consequence, the limit 
\bel{16mar17}  \lim_{\delta\downarrow 0}V^{1/2}P_n(H_0-E-i\delta)^{-1}P_nV^{1/2}=V^{1/2}((H_0-E)P_n)^{-1}V^{1/2}\ee
exists for $E\in(\gamma_n^-,\overline{{\mathcal E}}_n)$. The operator in the rhs of the last  equality  is a self-adjoint Hilbert-Schmidt  operator 
since $T(z)$ is Hilbert-Schmidt. Additionally, it  is continuous with respect to $E$. 

 Further, thanks to Lemma \ref{boch_lipt} and the definition of $L_n$, for $E\in (\gamma_n^-,\overline{{\mathcal E}}_n)$  we have the following  limit: 
\bel{Decomposeinband}
\ba{ll}\displaystyle{\lim_{\delta\downarrow 0}\sum_{l\in L_n}\int_\R\frac{G_l(k)^*G_l(k)}{E_l(k)-E -  i\delta}\,\rd k}&=\displaystyle{\sum_{l\in L_n}{\rm p.v.}\int_0^{\overline{\mathcal E}_l- \underline{\mathcal E}_l}\frac{G_l(\varrho_l(s))^*G_l(\varrho_l(s))}{s-\overline{\mathcal E}_l+E}\,\varrho_l'(s)\,\rd s}\\[1.5em]
&+i\pi\displaystyle{\sum_{l\in L_n}G_l(\varrho_l(\overline{\mathcal E}_l- E))^*G_l(\varrho_l(\overline{\mathcal E}_l- E))\varrho'_l(\overline{\mathcal E}_l- E),}
\ea
\ee
which combined with \eqref{16mar17} and \eqref{16mar17a} give us the existence of $\lim_{\delta\downarrow 0}T(E+i\delta)$.
Therefore, from the previous decomposition we get that
 $$\forall E\in (\gamma_{n}^{-},\overline{\mathcal E}_{n}), \quad {\rm Im} \,T(E+i0)=\pi\displaystyle{\sum_{l\in L_n}G_l(\varrho_l(\overline{\mathcal E}_l-E))^*G_l(\varrho_l(\overline{\mathcal E}_l-E))\varrho'_l(\overline{\mathcal E}_l-E),}$$
thus
 \bel{rankIm}
\forall E\in (\gamma_{n}^{-},\overline{\mathcal E}_{n}), \quad  {\rm Rank\,(Im} \,T(E+i0))=\#L_n.
  \ee
Furthermore, the continuity of $G_l^*G_l$ and of $\varrho_l'$ for any $l\in \N$,  imply the trace norm continuity of ${\rm Im}(T(\cdot+i0))$. \Bk We have  the results concerning ${\rm Im}(T(\cdot+i0)) $ altogether.  

Standard properties of the principal value show that the real part of the r.h.s. of \eqref{Decomposeinband} is continuous (in the operator norm) with respect to $E\in (\gamma_{n}^{-},\overline{\mathcal E}_{n})$. In consequence, ${\rm Re}\,T(\cdot+i0)$ is continuous in $ (\gamma_n^-,\overline{{\mathcal E}}_n)$.
 \end{proof}
To continue we need the following  lemma:
\begin{lemma}\label{estiPV}
For $\lambda\in (0,\overline{\cE_{n}}-\underline{\cE_{n}})$ and $\epsilon_{\lambda}\in (0,\lambda)$ we define : 
$$P(\lambda):={\rm p.v.}\displaystyle{\int_{\lambda-\epsilon_\lambda}^{\lambda+\epsilon_\lambda}
\frac{G_n(\varrho_n(s))^*G_n(\varrho_n(s))}{{\Bk s -\lambda \Bk}}\varrho_n'(s) \rd s}.$$
Then, there exists $C_{0}>0$ independent of $\lambda$ such that 
\bel{nov11_3}
\forall \lambda\in  (0,\overline{\cE_{n}}-\underline{\cE_{n}}), \quad \| P(\lambda)\|_{1}
\leq C_{0}\frac{\epsilon_\lambda}{\lambda^{2/M+1}}.\ee
\end{lemma}
\begin{proof}
Write
$$ P(\lambda)=\Bk{\rm p.v.}{\int_{\lambda-\epsilon_\lambda}^{\lambda+\epsilon_\lambda}\frac{G_n(\varrho_n(s))^*G_n(\varrho_n(s))}{{s -\lambda}}\varrho_n'(s)\,\rd s}
={\rm p.v.}\left(\mathcal{M}_1(\lambda)+\mathcal{M}_2(\lambda)\right).
$$
with 
\begin{equation*}
\left\{
\begin{aligned}
&\mathcal{M}_1(\lambda)=\displaystyle{\int_{0}^{\epsilon_\lambda}
\left({G_n(\varrho_n(\lambda+s))^*G_n(\varrho_n(\lambda+s))-G_n(\varrho_n(\lambda-s))^*G_n(\varrho_n(\lambda-s))}\right)\varrho'_j(\lambda+s) \frac{\rd s}{s}}\\
\\
&\mathcal{M}_2(\lambda)=\displaystyle{\int_{0}^{\epsilon_\lambda}
{G_n(\varrho_n(\lambda-s))^*G_n(\varrho_n(\lambda-s))\left(\varrho_n'(\lambda+s)-\varrho_n'(\lambda-s)\right)}\frac{\rd s}{s}.}\end{aligned}
\right.
\end{equation*}
Since both $\mathcal{M}_1(\lambda)$ and $\mathcal{M}_2(\lambda)$ are convergent integrals, there holds 
$${\rm p.v.}{\int_{\lambda-\epsilon_\lambda}^{\lambda+\epsilon_\lambda}\frac{G_n(\varrho_n(s))^*G_n(\varrho_n(s))}{{s -\lambda}}\varrho_n'(s)\,\rd s}
=\mathcal{M}_1(\lambda)+\mathcal{M}_2(\lambda).$$
From Lemma \ref{boch_lipt}, we have 
$$\|\mathcal{M}_1(\lambda)\|_{1}
\leq C \displaystyle{\int_{0}^{\epsilon_\lambda}\left|\frac{\varrho_n(\lambda+s)-\varrho_n(\lambda-s)}{s}\right|\left|\varrho'_n(\lambda+s)\right| \rd s}
$$
For each $\lambda>0$ and $\epsilon_\lambda\in(0,\lambda)$ fixed, we have that for  $s\in(0,\epsilon_\lambda), $ 
$\varrho_n(\lambda+s)-\varrho_n(\lambda-s)=\varrho'(c_{\lambda,s})s$
where $c_{\lambda,s}\in(\lambda-s,\lambda+s)\subset(\lambda-\epsilon_\lambda,\lambda+\epsilon_\lambda)\subsetneq(0,2\lambda).$
From \eqref{E:asymptvarrho}, we get
$|\varrho'_n(\lambda+s)|\leq C (\lambda+s)^{-1-1/M}< C (\lambda-s)^{-1-1/M}$
and $|\varrho'_n(c_{\lambda,s})|\leq C (c_{\lambda,s})^{-1-1/M} < C (\lambda-s)^{-1-1/M}$.
Then, we deduce 
$$
\ba{lll}
\|\mathcal{M}_1(\lambda)\|_{1}
& \leq&  C \displaystyle{\int_{0}^{\epsilon_\lambda}(\lambda-s)^{-2-2/M} \rd s}\\[1em]
&\leq&
\displaystyle{\tilde{C}\frac{\epsilon_\lambda}{\lambda^{2/M+1}}}.
\ea
$$
%
%
%
%
 For $\mathcal{M}_2$ we compute the asymptotics of $\varrho_{n}''(s)$ in 0 as in \eqref{E:asymptvarrho} and we get 
$$
\ba{lll}
\|\mathcal{M}_2(\lambda)\|_{1}
&\leq &C \displaystyle{\int_{0}^{\epsilon_\lambda}   \left|\frac{\varrho_n'(\lambda+s)-\varrho_n'(\lambda-s)}{s}\right|\rd s}\\[1em]

&\leq& \displaystyle{\tilde{C}\frac{\epsilon_\lambda}{\lambda^{2/M+1}}}.
\ea
$$\end{proof}
For $A\in(0,\overline{{\mathcal E}}_n-\underline{{\mathcal E}}_n]$ and $\lambda\in (-\infty,A)\setminus\{0\}$,  set 
\bel{dec12b}T_{n,A}(\lambda):={\rm p.v.}\int_0^A\frac{G_n(\varrho_n(s))^*G_n(\varrho_n(s))}{ s-\lambda}\varrho_n'(s)\rd s.
\ee
 Note that for $\lambda<0$, it is not necessary to define this operator as a principal value integral.

In the same manner as  $\gamma_{n}^{-}$ is  the nearest thresholds below   $\overline{\cE_{n}}$, define  $\gamma_{n}^{+}$ as
nearest thresholds above  $\overline{\cE_{n}}$. 
\begin{proposition}\label{pr2}
Assume the hypotheses of Theorem \ref{thssf}. Let $n\in \N$,  and    $A_-\in (0,\overline{\cE_{n}}-\gamma_{n}^{-})$,  $A_+\in (0,\gamma_{n}^{+}-\overline{\cE_{n}})$.
 For $\delta\in(0,1)$, take $I^{\pm}_\delta\subset (0,\delta A_\pm)$.  
Then, for  $\lambda\in I_\delta^\pm$, we can write
\bel{dec21}{\rm Re}\,T(\overline{\mathcal E}_n\pm\lambda+i0)=T_{n,A_\pm}(\mp\lambda)+\tilde{T}_n^\pm(\lambda),\ee
where the operators $T_{n,A_\pm}$, $\tilde{T}^\pm_n$ satisfy that there exists $C_0>0$ such that 
\bel{dec21a}\ba{lll}
||T_{n,A_+}(-\lambda)||_1&\leq C_0 |\lambda|^{-1}, & \forall\, \lambda\in I_\delta^+,\\
||T_{n,A_-}(\lambda)||_1&\leq C_0 (|\lambda|^{-1}+|\lambda|^{-2/M}), & \forall\, \lambda\in I_\delta^-,\ea
\ee 
\bel{20mar2017a}
\| \tilde{T}^\pm_n(\lambda)\|_{2}\leq C_0, \quad \forall\, \lambda\in I_\delta^\pm.\ee
\end{proposition} 
\begin{proof} Let us start with ${\rm Re}\,T(\overline{\mathcal E}_n-\lambda+i0)$. From Proposition \ref{pr1} and \eqref{dec12b} we can   write 
  \bel{DecompReT}
 {\rm Re}\,T(\overline{\mathcal E}_n-\lambda+i0)=T_{n,A_{-}}(\lambda)+\tilde{T}^-_n(\lambda),
 \ee
 with
\bel{20mar17}\ba{lll}\tilde{T}^{-}_n(\lambda)&=&\displaystyle{\int_{A_{-}}^{\overline{\mathcal E}_n- \underline{\mathcal E}_n}\frac{G_n(\varrho_n(s))^*G_n(\varrho_n(s))}{s-\lambda}\,\varrho_n'(s)\,\rd s}
\\
\\&+&\displaystyle{\sum_{l\in L_n, l\neq n}{\rm p.v.}\int_0^{\overline{\mathcal E}_l- \underline{\mathcal E}_l}\frac{G_l(\varrho_l(s))^*G_l(\varrho_l(s))}{s-\overline{\mathcal E}_l+\overline{\mathcal E}_n-\lambda}\,\varrho_l'(s)\,\rd s}\\
\\
&+&\displaystyle{V^{1/2}((H_0-\overline{\mathcal E}_n+\lambda)P_n)^{-1}V^{1/2}},\ea\ee
 where $L_{n}$ was defined in \eqref{D:Ln}.

 Now,  to  prove  \eqref{20mar2017a},
if   $\lambda$  is in $I_\delta^-$, 
the $\mathfrak{S}_1$-norm of  the first term of $\tilde{T}^{-}_n(\lambda)$ can be estimated as follows:
$$\ba{ll}
\left\|  \displaystyle{\int_{A_{-}}^{\overline{\mathcal E}_n- \underline{\mathcal E}_n}\frac{G_n(\varrho_n(s))^*G_n(\varrho_n(s))}{s-\lambda}\,\varrho_n'(s)\,\rd s}    \right\|_2&\leq \displaystyle{\sup_{s\in[A_{-},\overline{\mathcal E}_n- \underline{\mathcal E}_n)} \frac{1}{|s-\lambda|}
\int_{-\infty}^{\varrho_n(A_{-})}\| G_n(k)^*G_n(k)\|_2\,\rd k}\\[.6em]
&\leq\displaystyle{((1-\delta)A_{-})^{-1}\int_{-\infty}^{\infty}\| G_n(k)^*G_n(k)\|_2\,\rd k.}\ea$$   
Therefore, using  Lemma \ref{boch_lipt}, we can see that this term is uniformly bounded for $\lambda$ in $I_\delta^-$. 
A similar procedure can be performed to find a uniform bound for the second term of \eqref{20mar17}.

For the third term in \eqref{20mar17} we use 
$$\|V^{1/2}((H_0-\overline{\mathcal E}_n+\lambda)P_n)^{-1}V^{1/2}\|_2\leq \|V^{1/2}H_0^{-1}\|_2\,\|H_0((H_0-\overline{\mathcal E}_n+\lambda)P_n)^{-1}V^{1/2}\|,$$ 
and  the uniform boundedness of $\|H_0((H_0-\overline{\mathcal E}_n+\lambda)P_n)^{-1}\|$ for $\lambda\in I_\delta^-$. 



Finally, we write
\bel{16mar17b}\ba{ll}T_{n,A_{-}}(\lambda)=&{\rm p.v.}\displaystyle{\int_{\lambda/2}^{3\lambda/2}
\frac{G_n(\varrho_n(s))^*G_n(\varrho_n(s))}{{s-\lambda}}\varrho_n'(s)\,\rd s}\\[1em]
&+\displaystyle{\int^{\varrho_n(3\lambda/2)}_{\varrho_n(A_{-})}
\frac{G_n(k)^*G_n(k)}{E_n(k)-\overline{\mathcal{E}}_n+\lambda}\,\rd k}+
\displaystyle{\int_{\varrho_n(\lambda/2)}^{\infty}
\frac{G_n(k)^*G_n(k)}{E_n(k)-\overline{\mathcal{E}}_n+\lambda}\,\rd k.}
\ea\ee
 Using the monotonicity of $E_n(k)-\overline{\mathcal{E}}_n$, we get  $E_n(k)-\overline{\mathcal{E}}_n+\lambda>\frac{\lambda}{2}$ for all $k\in (\varrho_n(\lambda/2),+\infty)$ and $E_n(k)-\overline{\mathcal{E}}_n+\lambda<-\frac{\lambda}{2}$ for all $k\in (\varrho_n(A_{-}),\varrho_n(3\lambda/2))$. 
By the integrability of $||G_n(k)^*G_n(k)||$ we conclude that the $\mathfrak{S}_1$-norms of the   last two terms in \eqref{16mar17b} are controlled by $\lambda^{-1}$. The first term  in \eqref{16mar17b} can be  bounded applying Lemma \ref{estiPV} with $\epsilon_{\lambda}=\frac{\lambda}{2}$, obtaining that it is controlled by $\lambda^{-2/M}$. This  concludes the proof  of \eqref{dec21a}.
 The case ${\rm Re}\,T(\overline{\mathcal E}_n+\lambda+i0)$ is simpler: we prove in the same way that $\tilde{T}^{+}_{n}$ satifies \eqref{20mar2017a} for all $\lambda\in I_{\delta}^{+}$, and  we can read directly on \eqref{dec12b} that $T_{n,A_+}(-\lambda)$ is a well defined integral for $\lambda>0$, satsifying \eqref{dec21a}. 
  \end{proof}

\Bk

\section{Proof of the behavior of the spectral shift function}\label{proofasym_ssf}

In this section, we prove Theorems \ref{thssf}, \ref{thVcompact}  and \ref{thssf2}.
\subsection{Preliminary}
\subsubsection{Some results on the counting function of compact operators}
For $T$ being a compact self-adjoint operator set
 $$
n_{\pm}(r; T) : = {\rm Rank}\,\dP_{(r,\infty)}(\pm T);
$$
thus the functions $n_{\pm}(\cdot; T)$ are respectively the
counting functions of the positive and negative eigenvalues of the
operator $T$. If $T$ is compact but not necessarily
self-adjoint we will use the notation
$$
n_*(r; T) : = n_+(r^2; T^* T), \quad r> 0;
$$
thus  $n_{*}(\cdot; T)$ is the counting function of the singular
values of $T$. Evidently,
\bel{oct6_2}
n_*(r; T) = n_*(r; T^*), \quad n_+(r; T^*T) = n_+(r; T T^*), \quad r>0.
\ee
Besides, \Bk for $r_1, r_2>0$,  we have  the  Weyl inequalities
    \bel{weyl1}
    n_\pm(r_1 + r_2; T_1 + T_2) \leq n_\pm(r_1; T_1) + n_\pm(r_2; T_2),
    \ee
where  $T_j$, $j=1,2$, are
linear self-adjoint compact operators (see e.g. \cite[Theorem 9.2.9]{birsol}), as well as the Ky Fan inequality
    \bel{kyfan1}
    n_*(r_1 + r_2; T_1 + T_2) \leq n_*(r_1; T_1) + n_*(r_2;
T_2),
    \ee
    for compact but not necessarily self-adjoint $T_j$, $j=1,2$, (see e.g. \cite[Subsection 11.1.3]{birsol}).

Further, since for $T\in\mathfrak{S}_p,$ $ \| T \|_p  = \left( -\int_0^{\infty} r^p \,dn_*(r; T) \right)^{1/p} 
    $,
  the  Chebyshev-type estimate 
    \bel{cheby}
    n_*(r; T) \leq r^{-p} \|T\|_p^p
    \ee
    holds true for any $r > 0$ and $p \in [1, \infty)$.
\subsubsection{A useful representation of the SSF}
Let us recall the \emph{Pushnitski's representation of the SSF}.
 For $z \in \C^+$ and $V\geq 0$, we  have defined \Bk $T(z)= V^{1/2} (H_0-z)^{-1} V^{1/2}.$ As shown in Proposition \ref{pr1},  the norm limits
$\lim_{\delta\downarrow 0}T(E+i\delta) = T(E+i0)$ exist for every $E \in \R\setminus \cT_{H_0}$, provided that $V$ satisfies \eqref{cd2}. Then, by   \cite[Theorem 1.2]{pu}: For  almost all $E \in \R$, the SSF  $\xi(E;H_\pm,H_0)$ admits the representation  
    \bel{pssf}
    \xi(E;H_\pm,H_0) = \pm \frac{1}{\pi}\int_\R n_{\mp}(1;{\rm Re} \,T(E+i 0)+t\, {\rm Im} \, T(E +i0))\frac{dt}{1+t^2}.
    \ee

\subsection{Boundedness  and continuity of the SSF}

  Thanks to   representation \eqref{pssf} and \eqref{kyfan1}, for any $\epsilon \in (0,1)$
   \bel{22mar17}|\xi(\overline{\mathcal{E}}_n-\lambda;H_\pm,H_0)|\leq  n_*(1-\epsilon;{\rm Re}\,T(\overline{\mathcal{E}}_n-\lambda))+\frac{1}{\pi}\int_\R n_*(\varepsilon;t{\rm Im}\,T(\overline{\mathcal{E}}_n-\lambda))\frac{dt}{1+t^2}.\ee
Due to Proposition \ref{pr1}, the second term is uniformly bounded for  $\lambda$  in $(0,\overline{\cE}_{n}-\gamma_{n}^{-})$. 
Let $I_\delta^-$ be an  interval as in Proposition \ref{pr2}.  Using \eqref{dec21} and the  Ky Fan inequality \eqref{kyfan1} we get:
\bel{E:decoupetout}
\forall \epsilon \in (0,1), \quad n_*(1-\epsilon;{\rm Re}\,T(\overline{\mathcal{E}}_n-\lambda))\leq
\displaystyle{n_*\left(\frac{1-\epsilon}{2};T_{n,A}(\lambda)\right)+n_*\left(\frac{1-\epsilon}{2};\tilde{T}^-_{n}(\lambda)\right).}
\ee
Combining \eqref{cheby}, \eqref{20mar2017a} and \eqref{E:decoupetout} we can see that  for all $\epsilon \in (0,1)$ and all $\lambda\in I_\delta^-$
$$n_*(1-\epsilon;{\rm Re}\,T(\overline{\mathcal{E}}_n-\lambda))
\leq\displaystyle{\frac{4}{(1-\epsilon)^2}\|T_{n,A}(\lambda)\|_1+C_0.}
$$
By Proposition \ref{pr2}, if $K$ is any compact set contained in $I_\delta^-$, $\|T_{n,A}(\lambda)\|_1$ is uniformly bounded for $\lambda $ in $K$. Moreover, using \eqref{dec21a}, we get
$$n_*(1-\epsilon;{\rm Re}\,T(\overline{\mathcal{E}}_n-\lambda))
\leq C_0(1+\lambda^{-1}+\lambda^{-2/M}), \quad \forall \lambda\in I_\delta^-,$$
 implying  the first part of \eqref{SSFcontrolthresh}. The second part is similar.
 
Likewise, from \eqref{pssf} 
 \bel{E:decoupetout}
\forall \epsilon \in (0,1), \quad |\xi(\overline{\mathcal{E}}_n+\lambda;H_{-},H_0)|
\leq
\displaystyle{n_{+}\left(\frac{1-\epsilon}{2};T_{n,A}(-\lambda)\right)+O(1),}
\ee
and 
$$T_{n,A}(-\lambda)=-\int_{\varrho_n(A)}^\infty\frac{G_n(k)^*G_n(k)}{\overline{\mathcal{E}}_n-E_n(k)-\lambda},$$ which is clearly non-positive \Bk self adjoint operator, see \eqref{dec12b}. Therefore we have $n_{+}(\frac{1-\epsilon}{2},T_{n,A}(-\lambda))=0$, and we get \eqref{SSFboundedthresh}.



The continuity of the SSF follows from \cite[Corollary 2.6]{pu} together with Proposition \ref{pr1},  and is  similar to the proof of \cite[Proposition 2.5]{bpr}.

\subsection{Precise behavior at thresholds}

For  {$\lambda \in \R$} and $I \subset (0,\overline{{\mathcal E}}_n-\underline{{\mathcal E}}_n) \setminus \{\lambda\} $  let us introduce ${S}_{V}[ I]:L^2({\varrho_n(I)})\to L^2(\R^2)$ as the integral operator with kernel
$$
(2\pi)^{-1/2}{V}(x,y) ^{1/2} e^{iky}u_n(x;k)|E_n(k)-\overline{{\mathcal E}}_n+\lambda|^{-1/2},
$$
 that is: 
$$({S}_{V}[ I]f)(x,y)=(2\pi)^{-1/2}\int_{k\in I}{V}(x,y) ^{1/2} e^{iky}u_n(x;k)|E_n(k)-\overline{{\mathcal E}}_n+\lambda|^{-1/2}f(k) \rd k.$$

The next lemma link the behavior of the SSF at the threshold with $S_{V}[I]$.
\begin{lemma}
\label{L:isolatesing} Let $\lambda\in (0,\overline{\cE_{n}}-\underline{\cE_{n}})$,  $\epsilon_{\lambda}\in (0,\lambda)$ and    $A_-\in (0,\overline{\cE_{n}}-\gamma_{n}^{-})$. Then, for all $\delta\in (0,1)$:
\bel{avoir}\ba{ll}&
n_+ (1+\delta , S_V S_V ^* [(0, \lambda- \epsilon_\lambda ) ] )
{-  n_+} (\delta/2 , S_VS_V^*  [(\lambda+ \epsilon_\lambda,A)])+O\left(\frac{\epsilon_\lambda}{\lambda^{2/M+1}}\right) \\\leq&
-\xi(\overline{\mathcal E}_n -\lambda;H_-,H_0) \\\leq &
n_+ (1-\delta , S_V S_V ^* [(0, \lambda- \epsilon_\lambda ) ] )+O\left(\frac{\epsilon_\lambda}{\lambda^{2/M+1}}\right),\ea
\ee
 \bel{avoir2}\ba{ll}
 &n_+ (1+\delta, S_V S_V ^* [( \lambda+ \epsilon_\lambda,A)] )-n_+ (\delta/2 , S_V S_V ^* [(0, \lambda- \epsilon_\lambda ) ] )+O\left(\frac{\epsilon_\lambda}{\lambda^{2/M+1}}\right) \\ 
\leq &
\xi(\overline{\mathcal E}_n -\lambda;H_+,H_0) \\\leq &
{  n_+} (1-\delta , S_VS_V^*  [(\lambda+ \epsilon_\lambda,A)])
+O\left(\frac{\epsilon_\lambda}{\lambda^{2/M+1}}\right),\ea
\ee
and 
 \bel{avoir3}\ba{ll}&
n_+ (1+\delta , S_V S_V ^* [(0,A ) ] )
 \\\leq&
\xi(\overline{\mathcal E}_n +\lambda;H_+,H_0) \\\leq &
n_+ (1-\delta , S_V S_V ^* [(0, A ) ] ),\ea
\ee
where the remainders are uniform with respect to $\lambda$ and $\epsilon_{\lambda}$.
\end{lemma}
\begin{proof}
The proof of \eqref{avoir}, \eqref{avoir2} and \eqref{avoir3} follows the same lines. Let us prove \eqref{avoir}. \Bk  Proposition \ref{pr2}  together with the representation \eqref{pssf} and  the inequality \eqref{weyl1}  implies that for all $r_{1}<1<r_{2}$, as $\lambda \downarrow 0$:
\bel{dec12a}\ba{ll}
n_+(r_{2};T_{n,A}(\lambda))+O(1)
\leq -\xi(\overline{\mathcal E}_n -\lambda;H_-,H_0) \leq n_+(r_{1};T_{n,A}(\lambda))+O(1),
\ea\ee

Note the decomposition
$$
T_{n,A}(\lambda)=\mbox{p.v.}\displaystyle{\int_{\lambda -\epsilon_\lambda }^{\lambda+ \epsilon_\lambda}
\frac{G_{n}(\varrho_n(s))^*G_{n}(\varrho_n(s))}{{\lambda-s}}\varrho_n'(s)\rd s}- S_VS_V^*  [(\lambda+ \epsilon_\lambda,A)] + {S}_{V} {S}_{V} ^* [(0, \lambda-\epsilon_\lambda ) ].
$$ 
Then, the  Weyl inequalities \eqref{weyl1} together  with   Lemma \ref{estiPV}  imply that  for any $\eta \in (0,1)$: 
$$
n_+( r_{1}, T_{n,A}(\lambda) ) \leq n_+ (r_{1}(1-\eta) , S_V S_V ^* [(0, \lambda- \epsilon_\lambda ) ] )+O\left(\frac{\epsilon_\lambda}{\lambda^{2/M+1}}\right).
$$
Now, for $\delta\in (0,1)$ being given, we settle the parameters so that $1-\delta=r_{1}(1-\eta)$ and we get the upper bound by using \eqref{dec12a}.

In the same way, we have for any $\eta \in (0,1)$ and $a\in (0,1)$:
$$
n_+( r_{2}, T_{n,A}(\lambda) )\geq
n_+ (r_{2}(1+\eta) , S_V S_V ^* [(0, \lambda- \epsilon_\lambda ) ] ){-  n_+} (r_{2}a\eta , S_VS_V^*  [(\lambda+ \epsilon_\lambda,A)])
+O\left(\frac{\epsilon_\lambda}{\lambda^{2/M+1}}\right),
$$
and we get the lower bound by chosing the parameters so that $1+\delta=r_{2}(1+\eta)$ and $\frac{\delta}{2}=r_{2}a\eta$ and using \eqref{dec12a}.
\end{proof}
\Bk

\subsubsection{Proof of Theorem \ref{thVcompact}}

From \eqref{cheby}, \eqref{3mar17c} and \eqref{cd2}
\bel{3mar17b}\ba{ll}&
n_+(r,S_{V}S_{V}^*[\lambda, (0, \lambda- \epsilon_\lambda) ])=
n_+(r,S_V^* S_V  [(0, \lambda- \epsilon_\lambda ) ])\leq\frac{1}{r^2}\| S_V  [(0, \lambda- \epsilon_\lambda ) ]\|_2^2\\[.5em]
&=\displaystyle{\frac{1}{2\pi r^2}\int_{\varrho_{ n }(0,\lambda-\epsilon_\lambda)} \int_{\R^2}|E_n(k)-\overline{\mathcal{E}}_n+\lambda|^{-1} u_n(x;k)^2V(x,y)\rd x\rd y\rd k}\\[.5em]
&\displaystyle{\leq C \int_{(\varrho_n(\lambda-\epsilon_\lambda)}^\infty \int_{\R^2}|E_n(k)-\overline{\mathcal{E}}_n+\lambda|^{ -1 } e^{-\tau(x-x_{k})^2}\langle x,y\rangle^{-m}\,\rd x\rd y\rd k} \\[.5em]
\ea\ee
Making the change of variables $k=\varrho_n(s)$ and using \eqref{E:asymptvarrho} we can control the last integral as follows:
$$
\ba{ll} 
=  & \displaystyle{\int_0^{\lambda-\epsilon_\lambda} \int_{\R^2}|s-\lambda|^{-1}|\varrho'_n(s)|\langle x,y\rangle^{-m}e^{-\tau(x-a^{-1}(\varrho_n(s)))^2}\,\rd x\rd y\rd s}\\[.5em]
= C  & \displaystyle{\int_0^{\lambda-\epsilon_\lambda} |s-\lambda|^{-1}|\varrho'_n(s)| \int_{\R}\langle x\rangle^{1-m}e^{-\tau(x-a^{-1}(\varrho_n(s)))^2}\,\rd x\rd s}\\[.5em]
\leq  C  & \displaystyle{\int_0^{\lambda-\epsilon_\lambda} |s-\lambda|^{-1}|\varrho'_n(s)| \langle a^{-1}(\varrho_n(s))\rangle^{1-m}\rd s}\\[.5em]
\leq  C  & \displaystyle{\int_0^{\lambda-\epsilon_\lambda} |s-\lambda|^{-1}s^{-1-1/M} s^{(-1+m)/M}\rd s}.\\[.5em]
\ea
$$

Making $s=\lambda u$ and using that ${\frac{m-2-M}{M}}>0$: 
$$\int_0^{\lambda-\epsilon_\lambda} |s-\lambda|^{-1}s^{-1-1/M} s^{(-1+m)/M}\rd s=\lambda^{\frac{m-2-M}{M}}\int_{0}^{1-\frac{\epsilon_{\lambda}}{\lambda}}(1-u)^{-1}u^{\frac{m-2-M}{M}} \rd u \leq C \log(\tfrac{\epsilon_{\lambda}}{\lambda})\lambda^{\frac{m-2-M}{M}}$$
Setting $\epsilon_\lambda=\lambda^\theta$, we have proved:   
$$n_+(r,S_V^* S_V  [(0, \lambda- \epsilon_\lambda ) ])\leq C \log(\lambda)\lambda^{\frac{m-2-M}{M}},$$  
and if $ \theta > \frac{2}{M}+1$, using  Lemma \ref{L:isolatesing}, we conclude the proof of  Theorem \ref{thVcompact}  for the behavior of $\xi(\cE_{n}-\lambda,H_{-},H_{0})$ as $\lambda\downarrow 0$.

For he cases $\xi(\cE_{n}-\lambda,H_{+},H_{0})$ and $\xi(\cE_{n}+\lambda,H_{+},H_{0})$ we use \eqref{avoir2} and \eqref{avoir3}, respectively,  where each different term     is controlled exactly in the same way as above.
\begin{remark}
\label{R:now}
From the above inequalities we can see that the proof is still  true if we  
only assume \eqref{cd2}   valid only for  $x>0$.
\end{remark} 
 
 \subsubsection{Proof of Theorem \ref{thssf2}}
In what follows we will present the proof of  Theorem \ref{thssf2}  for the operator pair $(H_-,H_0)$. The case $(H_+,H_0)$ is similar and simpler, since the corresponding effective Hamiltonian $S_VS_V^* [(0,A)]$ does not contain  a principal value term (see \eqref{avoir3}).

 Consider  the class of symbols  $\mathcal{S}_p^q$ defined  as the set of  smooth functions $f$ such that 
for any $(\alpha,\beta) \in \Z_+^2$ the quantity
\bel{
b} n^{p,q}_{\alpha,\beta}(f):=\sup_{(x,y) \in \R^2}|\langle x\rangle^{p}\langle x,y\rangle ^{q+\alpha}\partial_y^\alpha \partial_x^\beta f(x,y)| \ee
is finite.

 For $f\in \mathcal{S}_p^q$ we define the operator  $Op^W(f)$ using  the Weyl quantization
$$(Op^{W}(f)u)(x):=\frac{1}{2\pi}\int_{\R^2}f\left(\frac{x+k}{2},y\right)e^{-i(x-k)y}u(k)\,\rd k\,\rd y,$$
where $u$ in the Schwartz space ${\mathcal S}(\R)$.

Define \Bk $Q_{V}:L^2(\R)\to L^2(\R^2)$ as the integral operator with kernel
\bel{Q_V}
(2\pi)^{-1/2}{V}(x,y)^{1/2} e^{iky}u_n(x,k).
\ee

Notice that for any  $I \subset (0,\overline{{\mathcal E}}_n-\underline{{\mathcal E}}_n) \setminus \{\lambda\} $,    
\bel{E:linkQS}
S_V[I]=Q_{V}\one_I(k)|E_n(k)-\overline{{\mathcal E}}_n+\lambda|^{-1/2}.
\ee
Since $V\in {\mathcal S}_0^m$,  we know from \cite[Lemma 5.1]{shi2}  that  for any $n\in\N$, the operator $Q_{V}^* \, Q_{V}:L^{2}(\R)\to L^{2}(\R)$ satisfies 
 \bel{shirai}Q_{V}^* \, Q_{V}=Op^W(w_n), \quad \mbox{with} \ \ w_n \in {\mathcal S}_0^m.\ee
Moreover the symbol $w_n$ satisfies
\bel{sep25}
w_{n}(x,y)={V}(a^{-1}(x),-y)+R_1(x,y),  \ \ \mbox{with} \  \ R_1\in {\mathcal S}^{m+1}_0. 
\ee

In the two following lemmas, we will control the terms appearing in Lemma \ref{L:isolatesing}, \eqref{avoir}. The core argument is to approximate $S_V S_V ^* [(0, \lambda- \epsilon_\lambda ) ]$ by pseudodifferential operators, whose symbols involve both $V$ and $a$.

\begin{lemma}\label{le6mar17} For any $r>0$,  $\theta>1$ and $p>m$ 
$$n_+(r,S_{V} S_{V} ^* [(\lambda+\lambda^\theta,A)])=O(\lambda^{-(p/M+\theta){/m}}), \quad \lambda \downarrow  0.$$
\end{lemma}

\begin{proof}
Due to  $\min_{k\in\varrho(\lambda+\lambda^\theta,A)}|\overline{{\mathcal E}}_n-E_n(k)-\lambda|=\lambda^\theta$, and using \eqref{E:linkQS},  we have the inequality $S_{V} S_{V} ^* [(\lambda+\lambda^\theta,A )]\leq (\lambda^\theta)^{-1}  Q_{V} \,\one_{\varrho(\lambda+\lambda^\theta,A)}(k)\,Q_{V}^*.$ Therefore,  
\bel{6mar17d}\ba{ll}n(r,S_{V} S_{V} ^* [(\lambda+\lambda^\theta,A ) ])&\leq n(r\lambda^\theta, Q_{V} \,\one_{\varrho(\lambda+\lambda^\theta,A)}(k)\,Q_{V}^*)\\
&\leq n(r\rho_\lambda\lambda^\theta, Q_{V}\langle k \rangle^{-p}\, Q_{V}^*)\\
&=n(r\rho_\lambda\lambda^\theta, \langle k\rangle^{-p/2}Q_{V}^*\, Q_{V}\langle k \rangle^{-p/2}),
\ea\ee
where $\rho_\lambda=(\sup_{k\in\varrho(\lambda+\lambda^\theta,A)}\langle k\rangle^{p})^{-1}$.
Next, \eqref{shirai} together with the composition formula, imply  that  $\langle k \rangle^{-p/2}Q_{V}^*\, Q_{V}\langle k \rangle^{-p/2}$ is a pseudodifferential operator whose Weyl symbol belongs to $S_p^m$. Since for any $f\in S_p^m$ holds true that $n_+(\lambda,Op^W(f))=O(\lambda^{-1/m})$  if $p>m$ (see Lemma 4.3 in \cite{brumi1}),  having in mind \eqref{E:asymptvarrho} we will obtain 
\bel{6mar17f}\ba{ll}n(\rho_\lambda\lambda^\theta, \langle\cdot\rangle^{-p/2}Q_{V}^{*}\, Q_{V}\langle\cdot\rangle^{-p/2})&=O((\rho_\lambda\lambda^\theta)^{-1/m})
\\&=O(\lambda^{-(p/M+\theta){/m}}), 	\quad \lambda \downarrow 0.\ea
\ee
Gathering \eqref{6mar17d} with  \eqref{6mar17f}  we get the result.
\end{proof}

Let $\varepsilon>0$ and take    a smooth function $\eta_\varepsilon$ with bounded derivatives  such that $0\leq \eta_\varepsilon(x) \leq 1$ for all $x \in \R$, $\eta_\varepsilon(x)=0$ for $x\leq-2\varepsilon$ and $\eta_\varepsilon(x)=1$ for $x\geq-\varepsilon$. We define then
$$
V_\eta(x,y):=\eta_\varepsilon(x) V(x,y).
$$
\begin{lemma}\label{effweyl} Assume that   $V\geq0$ satisfies \eqref{apr14} with $m>2$,   $N_0(\lambda, V)$ satisfies  \eqref{jul10a}, \eqref{jul10} and b satisfies \eqref{b_bound}, \eqref{E:HypB}. 
 Then, for any  $r>0$,  $\delta \in (0,1)$, $\theta>1$, $p>m$ 
$$
\ba{ll}&
n_+ \left((1+\delta) r \lambda , Op^W(V_{\eta}(a^{-1}(x),-y)\right)+O\left(\lambda^{-(p/M+\theta){/m}}\right) +o(\lambda^{-2/m})\\[.2em]
\leq&
n_{ + }( r,  S_V S_V ^* [(0, \lambda- \lambda^\theta ) ] ) ) \\[.2em]
\leq &
n_+ \left((1-\delta) r  \lambda , Op^W(V_{\eta}(a^{-1}(x),-y)\right)+O\left(\lambda^{-(p/M+\theta){/m}}\right)+o(\lambda^{-2/m}),\quad \lambda \downarrow 0.\ea
$$
\end{lemma}
\begin{proof}

%

 First, note that using \eqref{Q_V} we can define the operator $Q_{V_{\eta}}$,  
 and similarly  to \eqref{shirai}, $Q_{V_{\eta}}^* \, Q_{V_{\eta}}$ has a symbol $w_{\eta,n}$ in ${\mathcal S}_0^m$ which satisfies \eqref{sep25}, with $V_{\eta}$ instead of $V$. Therefore we will be allowed to use arguments similar to those of Lemma \ref{le6mar17}.

It is clear that \Bk  $S_{V}^* S_{V}-S_{V_\eta}^* S_{V_\eta}=
S^*_{V-V_\eta}S_{V-V_\eta}$, and  the function  $V-V_\eta$ is equal to zero for $x\geq-\varepsilon$.  Then, using Remark \ref{R:now},  the proof of Theorem \ref{thVcompact}  and  the   Weyl's inequalities, we conclude  that  for any $r>0$, $\delta \in (0,1)$ and $\theta>1$
\bel{may3} 
\begin{array}{ll}
&n_+(r(1+\delta);S_{V_\eta} S^*_{V_\eta} [ (0, \lambda- \lambda^\theta ) ])\\\leq&
n_+(r;S_{V} S_{V} ^* [(0, \lambda- \lambda^\theta ) ])\\ \leq &n_+(r(1-\delta);S_{V_\eta} S^*_{V_\eta} [ (0, \lambda- \lambda^\theta ) ]),\quad \lambda\downarrow 0.\end{array} \ee
Now, for $\delta \in(0,1)$ write \bel{12jun17}S_{V_\eta} S_{V_\eta} ^* [ (0, \lambda-\lambda^\theta ) ]=
S_{V_\eta} S_{V_\eta} ^* [(0, \delta\lambda ) ]+S_{V_\eta} S_{V_\eta} ^* [(\delta\lambda,\lambda-\lambda^\theta ) ].\ee  
Here, arguing as in the proof of Lemma \ref{le6mar17} we obtain 
\bel{6mar17}n_+(r, S_{V_\eta} S_{V_\eta} ^* [(\delta\lambda,\lambda-\lambda^\theta ) ])=O(\lambda^{-(p/M+\theta){/m}}).
\ee
Combining \eqref{may3}, \eqref{12jun17},  \eqref{6mar17} and using  \eqref{oct6_2}, we get
that \Bk for all $r>0$ and $\delta \in (0,1)$,
\bel{Interpres0} 
\begin{array}{ll}
&n_+(r(1+\delta);S_{V_\eta} S^*_{V_\eta} [ (0, \delta\lambda) ])+O(\lambda^{-(p/M+\theta){/m}})\\\leq&
n_+(r;S_{V} S_{V} ^* [(0, \lambda- \lambda^\theta ) ])\\ \leq &n_+(r(1-\delta);S_{V_\eta} S^*_{V_\eta} [ (0,\delta \lambda ) ])+O(\lambda^{-(p/M+\theta){/m}}),\quad \lambda\downarrow 0.\end{array} \ee

Next, to estimate $n_+(r,S_{V_\eta} S_{V_\eta} ^* [(0, \delta\lambda ) ])$  we use use \eqref{E:linkQS}:
\bel{3mar17a}\lambda^{-1}\one_{\varrho(0,\delta\lambda)}Q_{V_{\eta}}\, Q_{V_{\eta}}^*\one_{\varrho(0,\delta\lambda)}\leq
S_{V_\eta} S_{V_\eta} ^* [(0, \delta\lambda ) ]\leq ((1-\delta)\lambda)^{-1}Q_{V_{\eta}}\, Q_{V_{\eta}}^*.\ee
To get the lower bound, since 
$Q_{V_{\eta}}\,\one_{\varrho(0,\delta\lambda)} Q_{V_{\eta}}^*=  Q_{V_{\eta}}\,Q_{V_{\eta}}^*+
Q_{V_{\eta}}\,(1-\one_{\varrho(0,\delta\lambda)} )Q_{V_{\eta}}^*,
$
 by \eqref{weyl1}  
\bel{15may17b}\ba{ll}&n_+(r;\one_{\varrho(0,\delta\lambda)}Q_{V_{\eta}}\, Q_{V_{\eta}}^*\one_{\varrho(0,\delta\lambda)})=
n_+(r;Q_{V_{\eta}}\,\one_{\varrho(0,\delta\lambda)}Q_{V_{\eta}}^*)\\
\geq & n_+(r(1+\delta);Q_{V_{\eta}}\, Q_{V_{\eta}}^*)+n_{+}(r\delta,Q_{V_{\eta}}\,(1-\one_{\varrho(0,\delta\lambda)} )Q_{V_{\eta}}^*)\\
=&n_+(r(1+\delta);Q_{V_{\eta}}\, Q_{V_{\eta}}^*)+O(\lambda^{-(p/M+\theta){/m}}).\ea\ee
where the last estimate is obtained arguing again as in the proof of Lemma \ref{le6mar17}.

Then, \eqref{3mar17a}, \eqref{15may17b} and the min-max principle yield
\bel{dec19d}\ba{ll}&n_+(r(1+\delta)\lambda; Q_{V_{\eta}}^* \, Q_{V_{\eta}})+O\left(\lambda^{-(p/M+\theta)/m}\right)\\
\leq& n_+(r;S_{V} S_{V} ^* [ (0, \delta \lambda ) ])\\
\leq &
n_+(r(1-\delta)\lambda;\ Q_{V_{\eta}}^* \, Q_{V_{\eta}})+O\left(\lambda^{-(p/M+\theta){/m}}\right), \quad \lambda \downarrow 0.
\ea\ee
Using \eqref{sep25}, we apply
Lemma 4.5 in \cite{dauro} and we have 
$ n_+ (\lambda , Q_{V_{\eta}}^* \, Q_{V_{\eta}})=
n_+ (\lambda , Op^W(V_{\eta}(a^{-1}(x),-y)))+o(\lambda^{-2/m})$.
We combine this with \eqref{Interpres0} and \eqref{dec19d} to conclude.
\end{proof}
 Putting together Lemmas \ref{L:isolatesing}, \ref{le6mar17} and \ref{effweyl} we obtain that for any $r>0$ and $\delta\in(0,1)$  
$$
\ba{ll}&
n_+ ((1+\delta)r\lambda , Op^W(V_{\eta}(a^{-1}(x),-y))+O\left(\lambda^{-(p/M+\theta)/m}\right) +O\left(\lambda^{\theta-1-2/M}\right)+o(\lambda^{-2/m}) \\\leq&
 -\xi(\overline{\cE}_{n}-\lambda,H_{-},H_{0})  \\\leq &
n_+ ((1-\delta)r\lambda , Op^W(V_{\eta}(a^{-1}(x),-y) )+O\left(\lambda^{-(p/M+\theta)/m}\right)+O\left(\lambda^{\theta-1-2/M}\right)+o(\lambda^{-2/m}).\ea
$$
Then,  if  $m<M$, $m< p$ and $\theta>1$,  the biggest   term between   $\lambda^{\theta-1-2/M}$ and $\lambda^{-(p/M+\theta)/m}$ is the last one.  Minimizing  both $p$ and $\theta$,  this term becomes $\lambda^{-(1/m+1/M)}$. 
Therefore, we  obtain that for all $\delta\in(0,1)$
\bel{21mar17}
\ba{ll}
&n_+\left((1+\delta)\lambda;Op^W(V_\eta(a^{-1}(x),-y))\right)+o(\lambda^{-2/m})\\[.2em]
\leq& -\xi(\overline{\mathcal E}_n - \lambda;H_-,H_0)\\[.2em]
\leq &n_+\left((1-\delta)\lambda;Op^W(V_\eta(a^{-1}(x),-y))\right)+o(\lambda^{-2/m}),\quad \lambda \downarrow 0.
\ea
\ee
Next, it can be proved that $V_\eta(a^{-1}(x),-y)\in S_0^m$ and that  $N_0(\lambda,V)$ satisfying \eqref{jul10a}-\eqref{jul10} implies the same for $N(\lambda,V_\eta)$,
 where $N(\lambda,V_\eta):=\frac{1}{2\pi}vol\{(x,y)\in \R^2; V_\eta(x,y)> \lambda\}$ (\cite[Lemmas 3.7 and 3.8]{Mir16}). Then, \Bk by \cite[Theorem 1.3]{dauro}, we know that there exists $\nu>0$ such that 
\bel{15may17}n_+(\lambda,Op^W(V_\eta(a^{-1}(x),-y)))=N(\lambda,V_\eta(a^{-1}(x),-y))(1+O(\lambda^\nu)),\quad \lambda \downarrow 0.\ee Easy computations (\cite[Lemma 3.8]{Mir16}) show that 
\bel{21mar17a}\ba{c}
\displaystyle{\lim_{\lambda \downarrow 0}\frac{N(\lambda,V_\eta(a^{-1}(x),-y))}{ b_+\Bk N_0(\lambda,V)}=1;}\\[1em]
\displaystyle{\lim_{\delta\downarrow 0}\liminf_{\lambda\downarrow 0}\frac{N_0(\lambda(1+\delta),V)}{N_0(\lambda,V)}=\lim_{\delta\downarrow0}\limsup_{\lambda\downarrow 0}\frac{N_0(\lambda(1-\delta),V)}{N_0(\lambda,V)}=1.}\ea\ee
As a consequence,  \eqref{21mar17}, \eqref{15may17} and  \eqref{21mar17a} altogether imply \eqref{asym_ssf}.

\bibliographystyle{plain}
\bibliography{biblio,bibliopof}
\end{document}